\documentclass[letterpaper, 11pt]{amsart}

\usepackage{fullpage}
\usepackage{graphicx}
\usepackage{epstopdf}

\usepackage{tikz} \usetikzlibrary{arrows}

\usepackage[pdftex,bookmarks,colorlinks,breaklinks]{hyperref}  % PDF hyperlinks, with coloured links
\usepackage{color}
\definecolor{dullmagenta}{rgb}{0.4,0,0.4}   % #660066
\definecolor{darkblue}{rgb}{0,0,0.4}
\hypersetup{linkcolor=red,citecolor=blue,filecolor=dullmagenta,urlcolor=darkblue} % coloured links

\newcommand{\x}{{\mathbf x}}

\newtheorem{thm}{Theorem}[section]

\newtheorem{prop}[thm]{Proposition}

\theoremstyle{remark}

\newtheorem{conj}[thm]{Conjecture}

\begin{document}
\title{Computational Methods For \\
Extremal Steklov Problems}

\author{Eldar Akhmetgaliyev}
\address{Department of  Mathematics, 
Simon Fraser University, 
8888 University Drive
Burnaby, B.C.
Canada}
\email{eakhmetg@sfu.edu}
\thanks{}

\author{Chiu-Yen Kao}
\address{Department of Mathematical Sciences, 
Claremont McKenna College, 
Claremont, CA 91711}
\email{Chiu-Yen.Kao@claremontmckenna.edu}
\thanks{Chiu-Yen Kao is partially supported by NSF DMS-1318364.}

\author{Braxton Osting}
\address{Department of Mathematics, 
University of Utah, 
Salt Lake City, UT 84112, USA}
\email{osting@math.utah.edu}
\thanks{}

%\subjclass[2000]{Primary }
%    For articles to be published after 1 January 2010, you may use
%    the following version:
\subjclass[2010]{}


\date{\today}

\begin{abstract}
We develop a computational method for extremal Steklov eigenvalue problems and apply it to study the problem of maximizing the $p$-th Steklov eigenvalue as a function of the domain with a volume constraint. In contrast to the optimal domains for several other extremal Dirichlet- and Neumann-Laplacian eigenvalue problems, computational results suggest that the optimal domains for this problem are very structured. We reach the conjecture that the domain maximizing the $p$-th Steklov eigenvalue is unique (up to dilations and rigid transformations), has p-fold symmetry, and an axis of symmetry. The $p$-th Steklov eigenvalue has multiplicity 2 if $p$ is even and multiplicity 3 if $p\geq3$ is odd. 
\end{abstract}

\maketitle

\section{Introduction } \label{sec:intro}
Dedicated to the memory of Russian mathematician Vladimir Andreevich Steklov, a recent article in the Notices of the American Mathematical Society discuss his remarkable contributions to the development of science \cite{kuznetsov2014legacy}. One of his main contributions is on the study of the (second-order) Steklov eigenvalue problem, 
\begin{equation}
\label{eq: forward problem}
\left\{ \begin{array}{rcc}
\Delta u=0 & \mbox{in} & \Omega,\\
\partial_n u=\lambda u & \mbox{on} & \partial\Omega ,
\end{array}\right.
\end{equation}
named in his honor. Here, $\Omega\subset \mathbb R^d$ is a bounded  open set with Lipschitz boundary $\partial \Omega$, $\Delta$ is the Laplace operator, $\partial_n$ denotes the normal derivative, and $(\lambda,u)$ denotes the eigenpair. The Steklov spectrum is of fundamental interest as it coincides with the spectrum of the Dirichlet-to-Neumann operator $\Gamma\colon H^{\frac{1}{2}}(\partial\Omega)\to H^{-\frac{1}{2}}(\partial\Omega)$,  given by the formula $\Gamma u=\partial_{n}(\mathcal{H}u)$ where $\mathcal{H} u$ denotes the unique harmonic extension of $u\in H^{\frac{1}{2}}(\partial\Omega)$ to $\Omega$. It also arises in the study of sloshing liquids and heat flow. 

The Steklov spectrum is  discrete and we enumerate the eigenvalues in increasing order, 
$ 0=\lambda_0(\Omega)\le\lambda_1(\Omega)\le\lambda_2(\Omega)\le\ldots \to \infty $.
Weyl's law for Steklov eigenvalues, the asymptotic rate at which they tend to infinity,  is given by
$
\lambda_j \sim 2 \pi \left( \frac{j}{|\mathbb B^{d-1}| \ | \partial \Omega | }\right)^{\frac{1}{d-1}}
$ 
where $\mathbb B^{d-1}$ is the unit ball in $\mathbb R^{d-1}$ \cite{girouard2014spectral}. 
The eigenvalues also have a variational characterization, 
\begin{equation} \label{eq: lambda_k}
\lambda_{k}(\Omega)=
\min_{v\in H^{1}(\Omega)} \ 
\left\{ \frac{\int_{\Omega}\left|\nabla v \right|^{2}dx}{\int_{\partial\Omega}v^{2}ds}\colon \  
\int_{\partial\Omega}vu_{j}=0, \ 
j=0,\ldots,k-1\right\}.
\end{equation} 
where $u_j$ is the corresponding $j-$th eigenfunction. It follows from \eqref{eq: lambda_k} that Steklov eigenvalues satisfy the homothety property $\lambda_j(t \Omega) = t^{-1}\lambda_j(\Omega) $. We describe a number of previous results for extremal Steklov problems in Section \ref{sec:relWork}.

\subsection*{Statement of Results} In this short paper, we develop fast and robust computational methods for extremal Steklov eigenvalue problems.  
We apply these methods to the shape optimization problem 
\begin{equation} \label{eq:Ap}
\Lambda^{p\star} = \max_{\Omega \subset \mathbb R^2}  \  \Lambda_p(\Omega) 
\qquad \qquad \mathrm{where } \ \ \Lambda_p(\Omega) = \lambda_p(\Omega) \cdot  \sqrt{| \Omega| }. 
\end{equation}
Note that $\Lambda_p$ is invariant to dilations, so \eqref{eq:Ap} is equivalent to maximizing $\lambda_p(\Omega)$ subject to $|\Omega | = 1$. 
A computational study of this problem for values of $p$ between 1 and 101 suggests that the optimal domains are very structured and supports the following conjecture. 
\begin{conj} \label{optDom}
The maximizer, $\Omega^{p\star}$, of $\Lambda_p(\Omega)$ in \eqref{eq:Ap} is unique (up to dilations and rigid transformations), has $p$-fold symmetry, and an axis of symmetry. The $p$-th Steklov eigenvalue has multiplicity 2 if $p$ is even and multiplicity 3 if $p\geq 3$ is odd. 
\end{conj}
Furthermore, as described in Section~\ref{sec:NumRes}, the associated eigenspaces are also very structured. 
This structure stands in stark contrast with previous computational studies for extremal eigenvalue problems involving the Dirichlet- and Neumann-Laplacian spectra \cite{oudet2004numerical,osting2010,antunes2012,antunes2013optimization,osting2013minimal,osting2014minimal,kao2014maximal}. In particular, denoting the Dirichlet- and Neumann-Laplacian eigenvalues of $\Omega \subset \mathbb R^2$ by $\lambda^D(\Omega)$ and $\lambda^N(\Omega)$ respectively, computational results suggest that the optimizers for the following shape optimization problems do not seem to have structure:
\begin{align*}
&\min_{\Omega \subset \mathbb R^2} \ \lambda^D_p(\Omega) \cdot |\Omega|, 
\qquad
\min_{\Omega \subset \mathbb R^2} \ \sum_{p=k}^{k+\ell}  c_p \cdot \lambda^D_p(\Omega) \cdot |\Omega| \ \mathrm{ with } \   c_p \geq 0 \  \mathrm{ and } \  \sum_{p=k}^{k+\ell}  c_p = 1, \\
& \max_{\Omega \subset \mathbb R^2} \ \frac{\lambda^D_p(\Omega)}{ \lambda^D_1(\Omega)},   
\qquad 
 \min_{\Omega \subset \mathbb R^2} \ \lambda^D_p(\Omega) +  |\partial \Omega|, 
\qquad
\max_{\Omega \subset \mathbb R^2} \ \lambda^N_p(\Omega) \cdot |\Omega|. 
\end{align*}
The only exception that we are aware of is when the optimal value is attained by a ball or a sequence of domains which degenerates into the disjoint union of balls. 

For the problems listed above, we also note that the largest value of $p$ for which these previous studies have been able to access is $p\approx 20$. Here, we compute the optimal domains for $p = 100$ and $p=101$; our ability to compute optimal domains for such large values of $p$ arises from (1) a very efficient and accurate Steklov eigenvalue solver and (2) a slight reformulation of the eigenvalue optimization problem that significantly reduces the number of eigenvalue evaluations required. 

\subsection*{Outline} 
In Section~\ref{sec:relWork}, we review some related work.
Computational methods are described in Section~\ref{sec:CompMeth}. 
Numerical experiments are presented in Section~\ref{sec:NumRes} and we conclude in Section~\ref{sec:disc} with a brief discussion. 

\section{Related Work} \label{sec:relWork}
Here we briefly survey some related work; a more comprehensive review can be found in  \cite{girouard2014spectral} and a  historical viewpoint with applications in \cite{kuznetsov2014legacy}. 

In 1954, R. Weinstock proved that the disk maximizes the first non-trivial  Steklov eigenvalue of
\begin{equation}
\label{eq: forward inhom problem}
\left\{ \begin{array}{rcc}
\Delta u=0 & \mbox{in} & \Omega,\\
\partial_n u=\lambda \rho u & \mbox{on} & \partial\Omega,
\end{array}\right.
\end{equation}
among \emph{simply-connected} planar domains with a fixed total mass $M(\Omega)=\int_{\partial \Omega} \rho(s)ds$ where $\rho$ is an $L^\infty(\partial \Omega)$ non-negative weight function on the boundary, referred to as the ``density'' \cite{weinstock1954inequalities,girouard2010shape}. It remains an open question for non-simply-connected bounded planar domains \cite{girouard2014spectral}. In 1974, J. Hersch, L. E. Payne, and M. M. Schiffer proved a general isoperimetric inequality for simply-connected planar domains which, in a special case, can be expressed
\begin{equation} \label{HPS} 
\sup \{ \lambda_n(\Omega)\cdot M(\Omega) \colon \Omega \subset \mathbb R^2 \}  \  \leq  \ 2 \pi n, \qquad n \in \mathbb N. 
\end{equation}
In \cite{girouard2010shape}, A. Girouard and I. Polterovich provided an alternative proof based on complex analysis to show that disk maximizes the first non-trivial  Steklov eigenvalue. Furthermore, they proved that the maximum of second eigenvalue is not attained in the class of simply-connected domain instead by a sequence of simply-connected domains degenerating to a disjoint union of two identical disks. In \cite{girouard2010}, A. Girouard and I. Polterovich proved that the bound in (\ref{HPS}) is sharp and attained by a sequence of simply-connected domains degenerating into a disjoint union of $n$ identical balls.   

An extension of R. Weinstock's result to arbitrary Riemannian surfaces $\Sigma$ with genus $\gamma$ and $k$ boundary components was given by A. Fraser and R. Schoen in \cite{fraser2011first}. The inequality
\begin{equation} \label{Riemannian estimate} 
\lambda_1(\Sigma) \cdot |\partial \Sigma | \le 2(\gamma + k) \pi
\end{equation}
derived therein reduces to R. Weinstock's result for $\gamma = 0$ and $k=1$ and the bound is sharp. However, for $\gamma = 0$ and $k=2$, the bound is not sharp. See \cite{fraser2011first} for a better upper bound on annulus surfaces.  In \cite{colbois2011}, it is proven that there exists a constant $C = C(d)$, such that for every bounded domain $\Omega \subset \mathbb R^{d}$, 
\begin{equation}
\label{eq:LinearGrowth}
\lambda_k (\Omega) \cdot | \partial \Omega |^{\frac{1}{d-1}} \leq C k^{\frac{2}{d}}, \qquad \qquad k \geq 1.  
\end{equation} 
A generalization for Riemannian manifolds is also given. 

Other objective functions depending on Steklov eigenvalues were also considered. 
In \cite{hersch1974}, J. Hersch, L. E. Payne, and M. M. Schiffer proved that the minimum of $ \sum_{k=1}^n \lambda_k^{-1}(\Omega)$ is attained when $\Omega$ is a disk for both perimeter and area constraints. This result is generalized to arbitrary dimensions in \cite{brock2001isoperimetric}.  
In \cite{dittmar2004sums}, it is proven that sums of squared reciprocal Steklov eigenvalues, $\sum_{k=1}^{\infty} \lambda_k^{-2}(\Omega)$, for simply-connected domains with a fixed perimeter is minimized by a disk. The sharp isoperimetric upper bounds were found for the sum of first $k$-th eigenvalues, partial sums of the spectral zeta function, and heat trace for starlike and simply-connected domains  using quasiconformal mappings to a disk \cite{Girouard2015}.

\medskip

There are also a few computational studies of extremal Steklov problems. The most relevant is recent work of B. Bogosel \cite{Bogosel2015}. This paper is primary concerned with the development of methods based on fundamental solutions to compute the Steklov, Wentzell, and Laplace-Beltrami eigenvalues. This method was used to demonstrate that the ball is the minimizer for a variety of shape optimization problems. The author also studies the problem of maximizing the first five Wentzell eigenvalues subject to a volume constraint, for which \eqref{eq:Ap} is a special case.  Shape optimization problems for Steklov eigenvalues with mixed boundary conditions have also been studied \cite{bonder2007optimization}. 

\section{Computational Methods} \label{sec:CompMeth}
\subsection{Computation of Steklov Eigenvalues}
We consider the Steklov eigenvalue problem~\eqref{eq: forward problem} where the domain $\Omega$ is simply-connected with smooth boundary $\partial \Omega$. Without loss of generality we assume that $\partial \Omega$  possesses a $2\pi$-periodic counterclockwise parametric representation of the form
\begin{equation*}\label{steklov_curve_parametrization}
x(t) = (x_1(t), x_2(t)), \quad 0 \leq t \leq 2\pi.
\end{equation*}
Use of integral equation methods for \eqref{eq: forward problem} leads directly
upon discretization to a matrix eigenvalue problem \cite{huang2004mechanical,cheng2012nystrom}. In order to avoid the inclusion of hypersingular operators we use eigenfunction representations based on a single layer potential.  The eigenfunction $u(x)$ is represented using a single layer potential, $\varphi$, with a slight modification to ensure uniqueness of the solution
\begin{equation}
\label{steklov_single_layer_modified}
u(x) = \int_{ \partial \Omega} \Phi (x-y)(\varphi(y) - \overline{\varphi}) ds(y) + \overline{\varphi},
\end{equation}
where $\displaystyle \Phi(x) = \frac{1}{2\pi}\log|x|$ and  $\overline{\varphi}  = \frac{1}{|\partial \Omega|}\int_{\partial \Omega} \varphi(y)  ds(y)$. 
See~\cite[Theorem 7.41]{kress1999} for the proof that the corresponding boundary operator is bijective. 
%The proposed Steklov solver is therefore based on representation~\eqref{steklov_single_layer_modified}. 
Taking into account well-known expressions (see e.g.~\cite{kress1999}) for the
jump of the single layer potential and its normal derivative across
$\partial \Omega$, the eigenvalue problem~\eqref{eq: forward problem} 
reduces to the integral eigenvalue equation for $(\lambda,\varphi)$, 
\begin{equation}
\label{steklov_integral_equation_system}
A[\varphi] = \lambda B[\varphi]. 
\end{equation}
Here, the boundary operators $A$ and $B$ are defined
\begin{align*}
A[\varphi](x) &:=  \int_{\partial \Omega} \frac{\partial\Phi(x-y)}{\partial n(x)}(\varphi(y) - \overline{\varphi}) ds(y) + \frac{1}{2}(\varphi(x) - \overline{\varphi})\\
B[\varphi](x) &:= \lambda\left(\int_{\partial \Omega} \Phi (x-y)(\varphi(y) - \overline{\varphi}) ds(y)+ \overline{\varphi} \right).
\end{align*}
In the cases considered in this paper, the Steklov eigenfunctions $u_k$ and the corresponding
densities $\varphi_k$ are smooth functions. These problems can thus be 
treated using highly effective spectrally-accurate
methods~\cite{COLTON:1998,kress1999}  based on explicit resolution of logarithmic
singularities and a Fourier series approximation of the density.  
To construct a spectral method for approximation of the integral operators in
\eqref{steklov_integral_equation_system}, we use a  Nystr\"om discretization of the explicitly parametrized boundary $\partial \Omega$.  
This spectral approximation of the integral equation system yields a
generalized matrix eigenvalue problem of the form
\begin{equation} 
{\tt A} {\tt X} = {\tt \Lambda} {\tt B} {\tt X}, 
\end{equation}
which can be solved numerically by means of the QZ-algorithm (see~\cite{golub2012matrix}). More details about this method can be found in \cite{akhmetgaliyev2016fast,AkhmetgaliyevNigamBruno}. 

\begin{figure}[t!]
\begin{center}
\includegraphics[width=.45\linewidth]{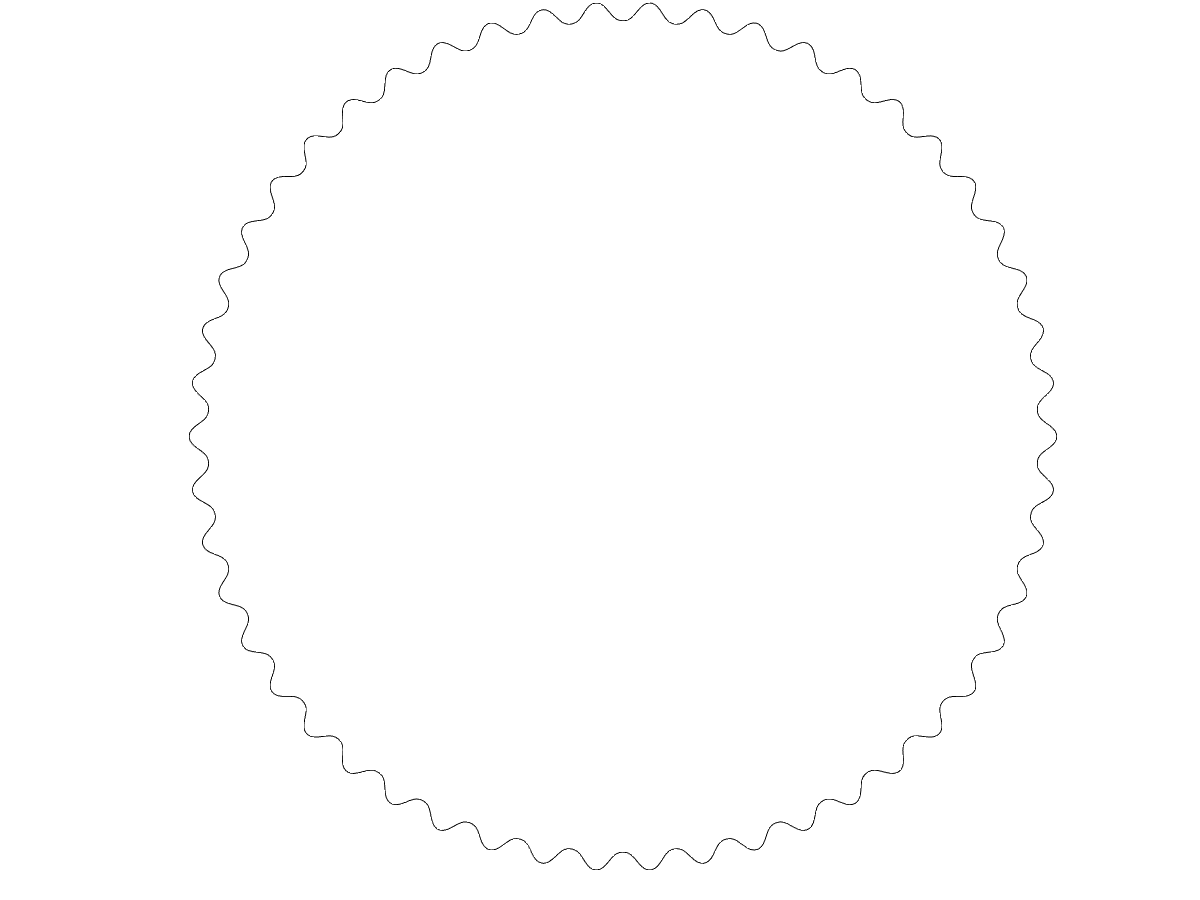} 
\includegraphics[width=.45\linewidth]{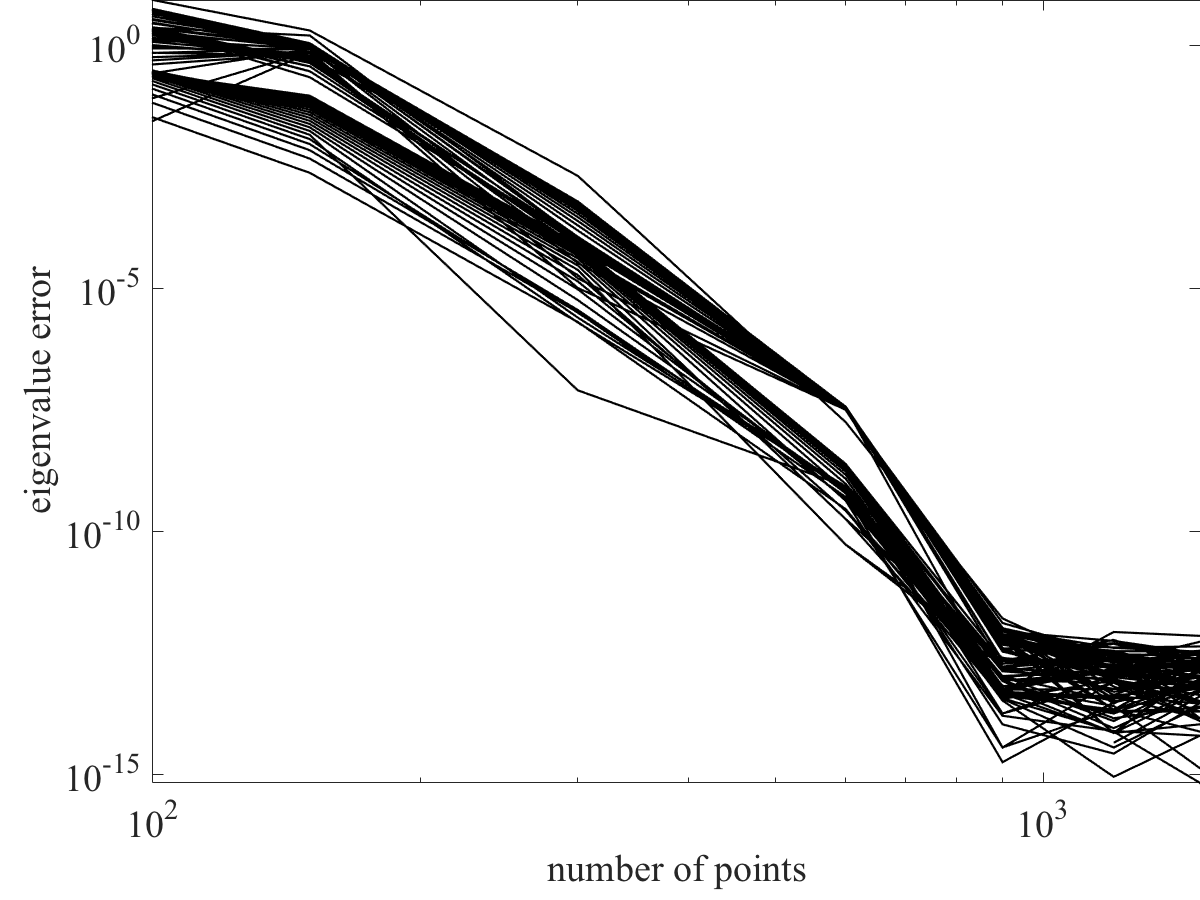}
\end{center}
\caption{ A log-log convergence plot of the first 100 eigenvalues for the domain on the left as the number of interpolation points increases. }
\label{fig:conv}
\end{figure}

In Figure~\ref{fig:conv}, we demonstrate the spectral  convergence of this boundary integral method.  In the left panel we depict a domain that was obtained as an optimizer for the $50$-th Steklov eigenvalue. In polar coordinates, this domain is given by  $\{(r,\theta)\colon r< R(\theta) \}$ where
{\footnotesize
$$ 
R(\theta) =  2.5 
+  0.057475351612645 \cdot \cos(50 \ \theta) 
+  0.002675998736772 \cdot \cos(100 \ \theta)
-   0.002569287572637 \cdot \cos(150 \ \theta).
$$}
In the right panel, we display a log-log convergence plot of the first $100$ Steklov eigenvalues of the domain in the left panel, as we increase the number $n$ of interpolation points. For ground-truth, we used $n = 1800$.

\subsection{Eigenvalue Perturbation Formula}The following proposition gives the Steklov eigenvalue perturbation formula, which can also be found in  \cite{dambrine2014extremal}. 

\begin{prop} \label{prop:EigPerpFormula} Consider the perturbation $x\mapsto x + \tau v$ and write $c = v\cdot \hat n$ where $\hat n$ is the outward unit normal vector. Then a simple (unit-normalized) Steklov eigenpair $(\lambda,u)$  satisfies the perturbation formula 
\begin{equation} \label{eq:EigPerpFormula}
\lambda' = \int_{\partial \Omega} \left(|\nabla u |^2 - 2 \lambda^2 u^2 -  \lambda \kappa u^2    \right)  c \ dx 
\end{equation}
\end{prop}
\begin{proof}
Let primes denote the shape derivative. 
From the identity $\lambda = \int_\Omega |\nabla u |^2 \ dx$, we compute
\begin{align*}
\lambda' &= 2\int_\Omega \nabla u \cdot \nabla u' \ dx + \int_{\partial \Omega}  |\nabla u |^2  c \ dx 
&& \textrm{(shape derivative)} \\
&= - 2 \int_{\Omega} (\Delta u) u'  \ dx + 2 \int_{\partial \Omega} u_n u'  \ dx + \int_{\partial \Omega}  |\nabla u |^2  c \ dx 
&& \textrm{(Green's identity)} \\
&= 2\lambda \int_{\partial \Omega} u u'  \ dx + \int_{\partial \Omega}  |\nabla u |^2  c \ dx
&& \textrm{(Equation \eqref{eq: forward problem})}. 
\end{align*}
Differentiating  the normalization equation, $ \int_{\partial \Omega} u^2  \ dx= 1$, we have that 
\begin{align*}
 \int_{\partial \Omega} u u' \ dx = -  \int_{\partial \Omega} \left(u u_n + \frac{ \kappa }{2} u^2\right) c \ dx 
 =  - \int_{\partial \Omega} \left( \lambda + \frac{ \kappa }{2} \right)u^2 \  c \ dx . 
 \end{align*}
 Putting these two equations together, we obtain \eqref{eq:EigPerpFormula}. 
 \end{proof}

\subsection{Shape Parameterization }

We consider domains of the form 
\begin{equation} \label{eq:BndyExpan}
\Omega = \{ (r,\theta)\colon\, 0 \leq r < \rho(\theta) \}, \qquad \text{where} \  \rho(\theta) = \sum_{k=0}^m a_k \cos(k\theta) +  \sum_{k=1}^m b_k \sin(k \theta). 
\end{equation}
The velocities  corresponding to a perturbation of the $k$-th cosine and sine coefficients are given by 
\begin{align*}
\frac{\partial \x(\theta)}{\partial a_k} \cdot \hat n(\theta) = \frac{\rho(\theta) \cos( k \theta)}{ \sqrt{\rho^2(\theta) + [\rho'(\theta)]^2}} 
\qquad \textrm{and} \qquad
\frac{\partial \x(\theta)}{\partial b_k} \cdot \hat n(\theta) = \frac{\rho(\theta) \sin( k \theta)}{ \sqrt{\rho^2(\theta) + [\rho'(\theta)]^2}} . 
\end{align*}
The derivative of  Steklov eigenvalues with respect to Fourier coefficients can be obtained using Proposition \ref{prop:EigPerpFormula}.

\subsection{Optimization Method}
We apply gradient-based optimization methods to minimize spectral functions of Steklov eigenvalues, such as  \eqref{eq:Ap}. The gradient of a simple eigenvalue is provided in Proposition~\ref{prop:EigPerpFormula}. While Steklov eigenvalues are not differentiable when they have multiplicity greater than one,  in practice, eigenvalues computed numerically that approximate the Steklov eigenvalues of a domain are always simple. This is due to discretization error and finite precision in the domain representation. Thus, we are faced with the problem of maximizing  a function that we know to be non-smooth, but whose gradient  is well-defined at points in which we sample. 

To compute solutions to the eigenvalue optimization problem \eqref{eq:Ap}, we (trivially)  reformulate the problem as a minimax problem, 
$$
\max_{\Omega \subset \mathbb R^2}  \  \min_{j\colon p \leq j\leq p-1+m } \Lambda_j(\Omega), \qquad \qquad m \geq 1. 
$$ 
This minimax problem can be numerically solved using Matlab's \verb+fminimax+ function which further reformulates the optimization problem 
\begin{align*} \label{eq:fminimax}
\max_{\Omega \subset \mathbb R^2}  & \  t \\
\mathrm{s.t.} & \ \Lambda_j(\Omega) \geq t, \qquad j = p, p+1, \ldots, p-1 +m. 
\end{align*}
and solves this problem using nonlinear constrained optimization methods. We choose $m$ to be the (expected) multiplicity of the eigenvalue at the optimal solution. For \eqref{eq:Ap}, we find this method to be more effective then using the BFGS quasi-Newton method directly, as reported in other computational studies of extremal eigevalues  \cite{osting2010,antunes2012,osting2013minimal,osting2014minimal,kao2014maximal}.

%We apply gradient-based methods to minimize spectral functions of Steklov eigenvalues. The gradient of a simple eigenvalue is provided in Proposition~\ref{prop:EigPerpFormula}. While Steklov eigenvalues are not differentiable when they have multiplicity greater than one,  in practice, eigenvalues computed numerically that approximate the Steklov eigenvalues of a domain are always simple. This is due to discretization error and finite precision. Thus, we are faced with the problem of maximizing  a function that we know to be non-smooth, but whose gradient  is well-defined at points in which we sample. For a variety of such non-smooth problems, the BFGS quasi-Newton method with an inexact line search has proven to be very effective \cite{LewisOverton2013}, but the convergence theory remains sparse. In particular, for this problem, a gradient descent  algorithm will generate a sequence of domains where the $k$-th and $(k+1)$-th eigenvalues will converge towards each other. The sequence will become ``stuck'' at this point and the objective function values will be relatively small compared to the optimal value.  As reported in other computational studies of extremal eigenfunctions  \cite{antunes2012,kao2014maximal,osting2010,osting2013minimal,osting2014minimal}, for this problem we observe that a BFGS approximation to the Hessian avoids this phenomena. 

\section{Numerical Results} \label{sec:NumRes}
In this section, we apply the computational methods developed in Section \ref{sec:CompMeth} to the Steklov eigenvalue optimization problem  \eqref{eq:Ap}. 
The  methods are implemented in Matlab and numerical results are obtained on a 4-core 4 GHz Intel Core i7 computer with 32GB of RAM. 
 Unless specified otherwise, we initialize with randomly chosen Fourier coefficients and the number of interpolation points used is $6\cdot p\cdot m$ where $p$ is the eigenvalue considered and $m$ is largest free Fourier coefficient for the domain. 

\begin{figure}[t!]
\begin{center}
\vspace{-.5cm}
\includegraphics[width=.3\linewidth]{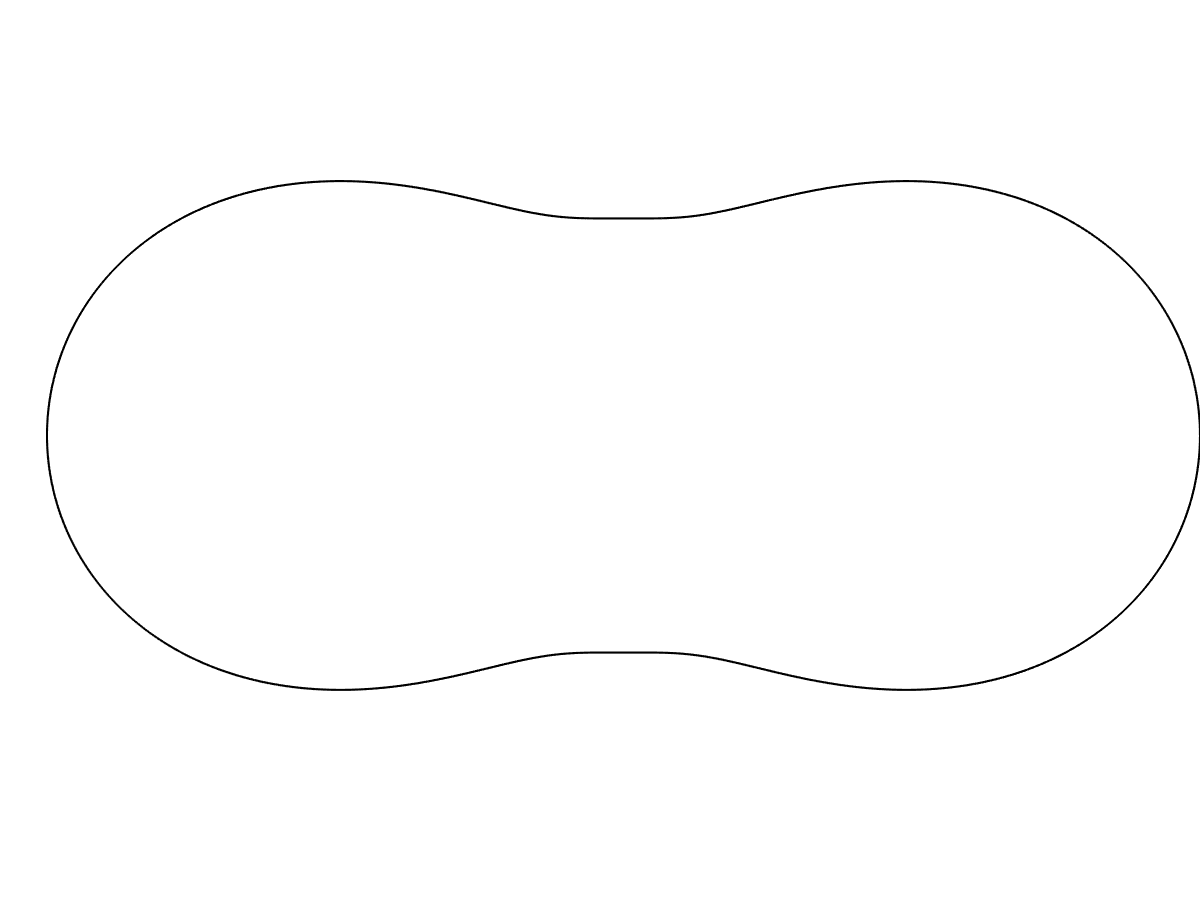}
\includegraphics[width=.3\linewidth]{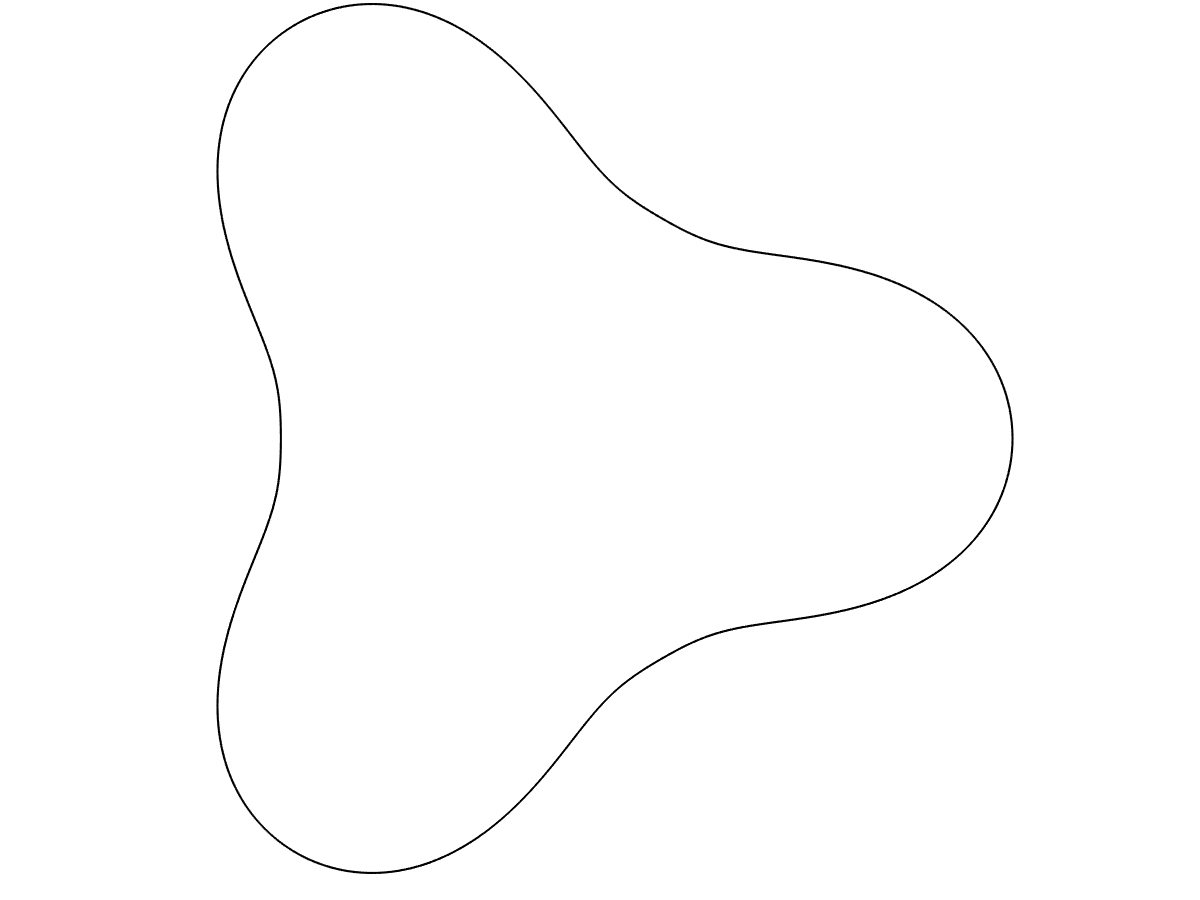}
\includegraphics[width=.3\linewidth]{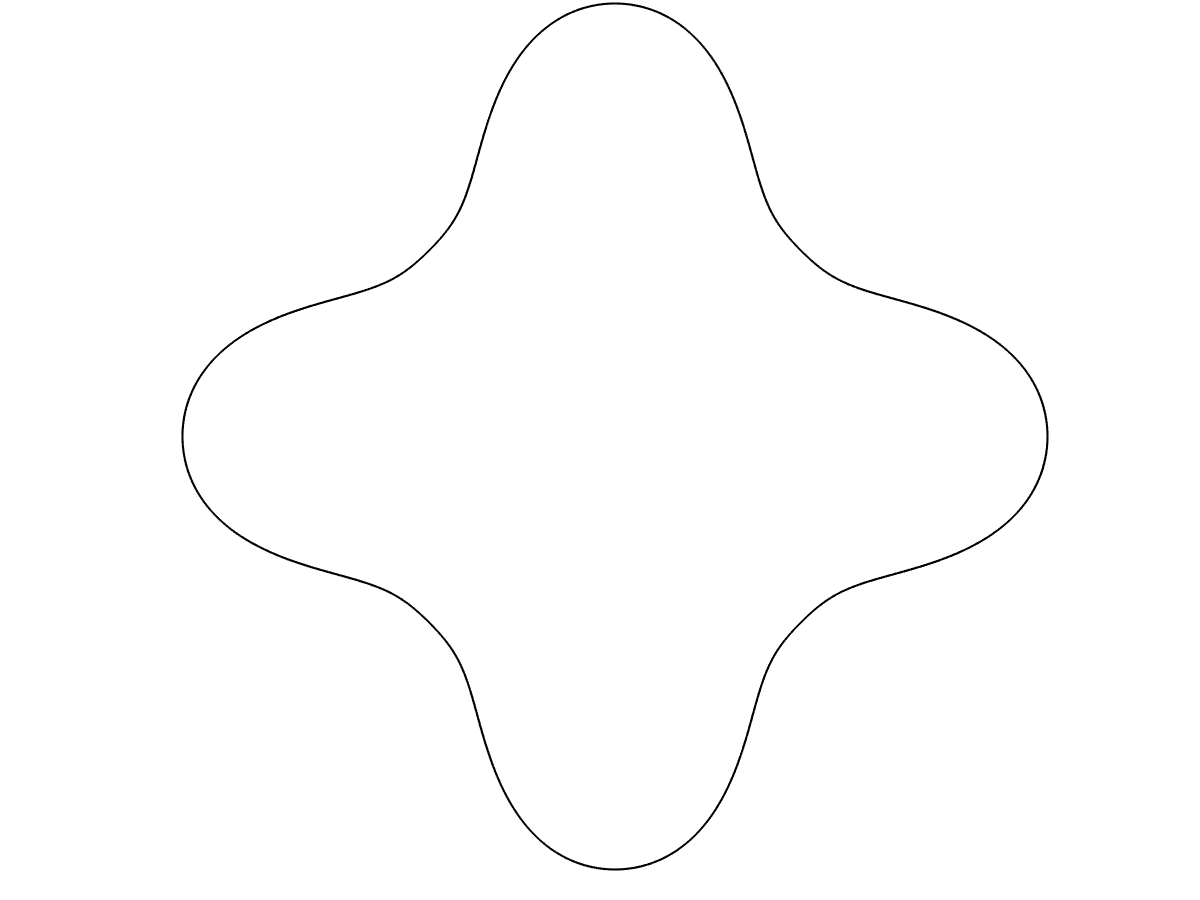}
\includegraphics[width=.3\linewidth]{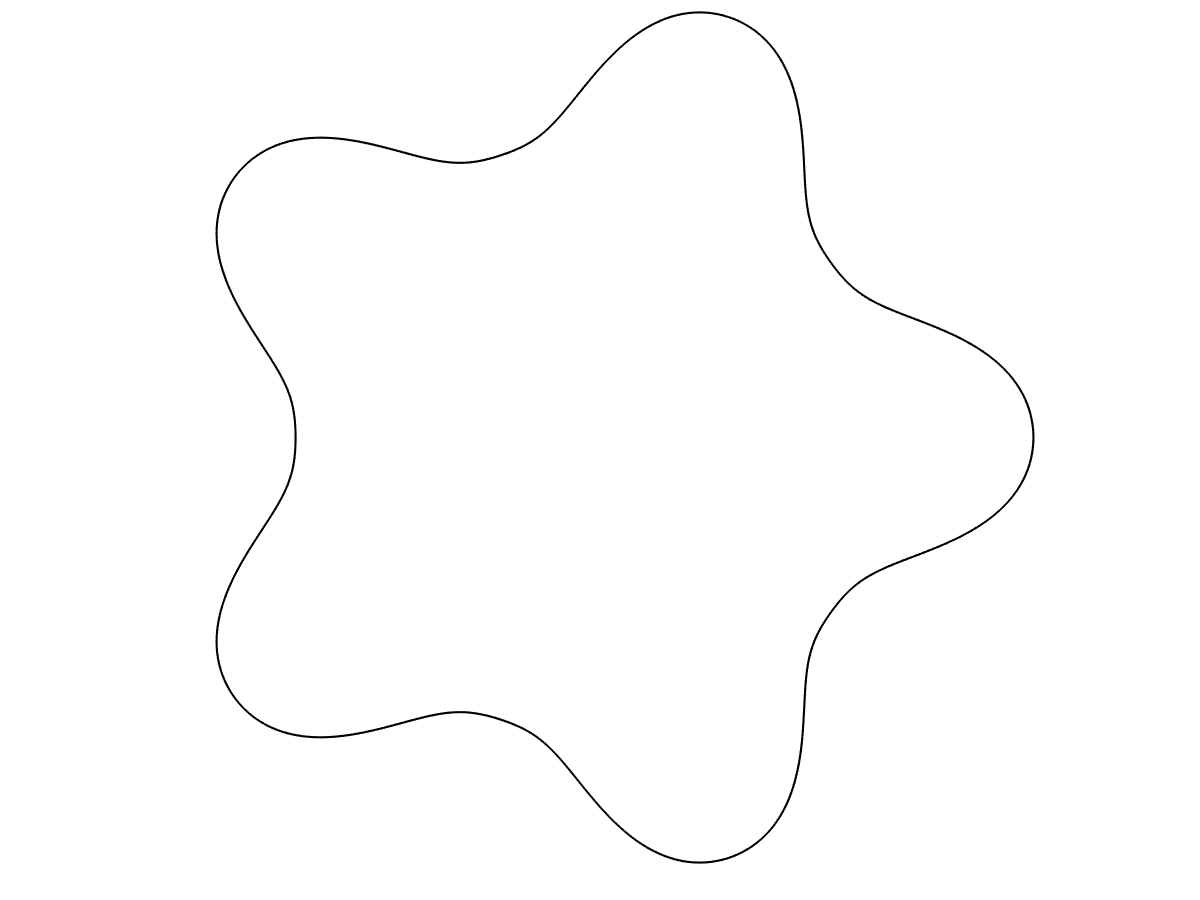}
\includegraphics[width=.3\linewidth]{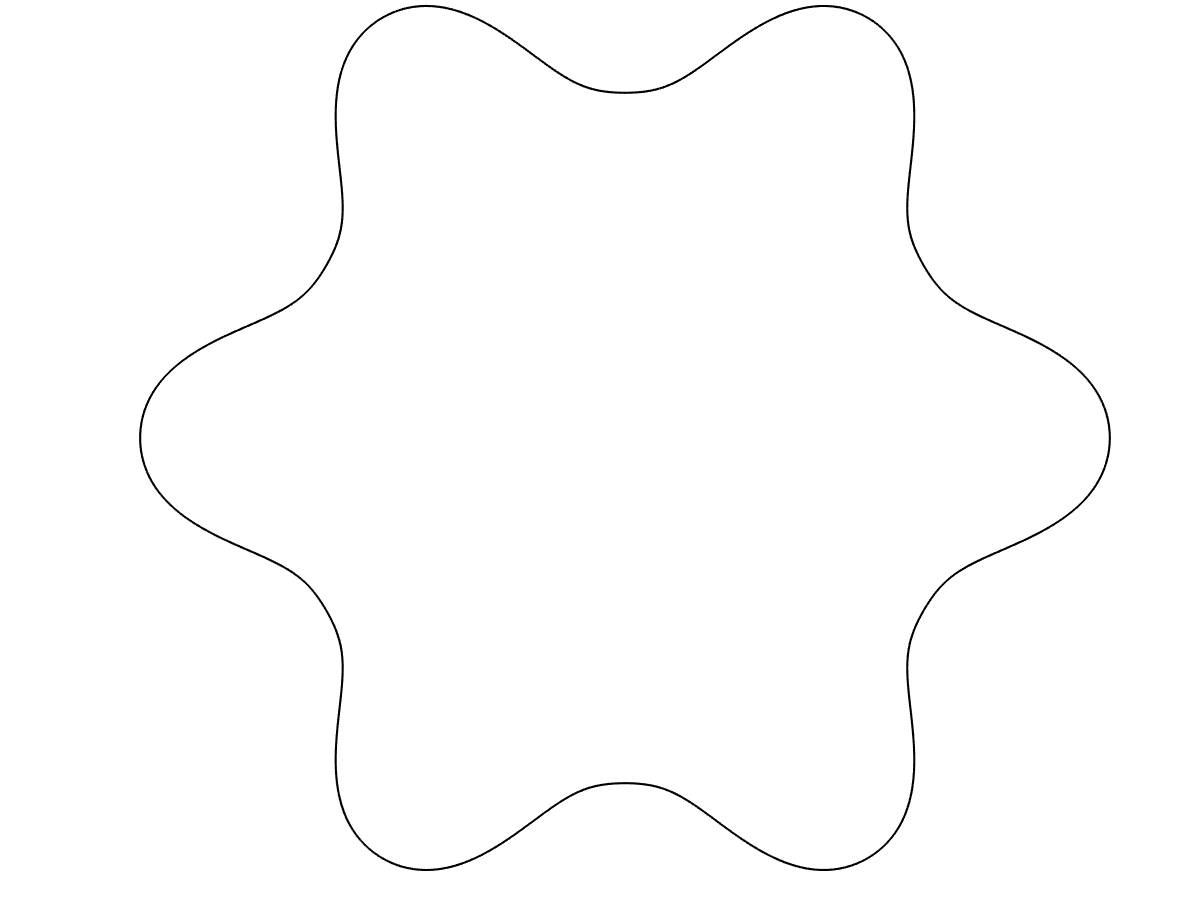}
\includegraphics[width=.3\linewidth]{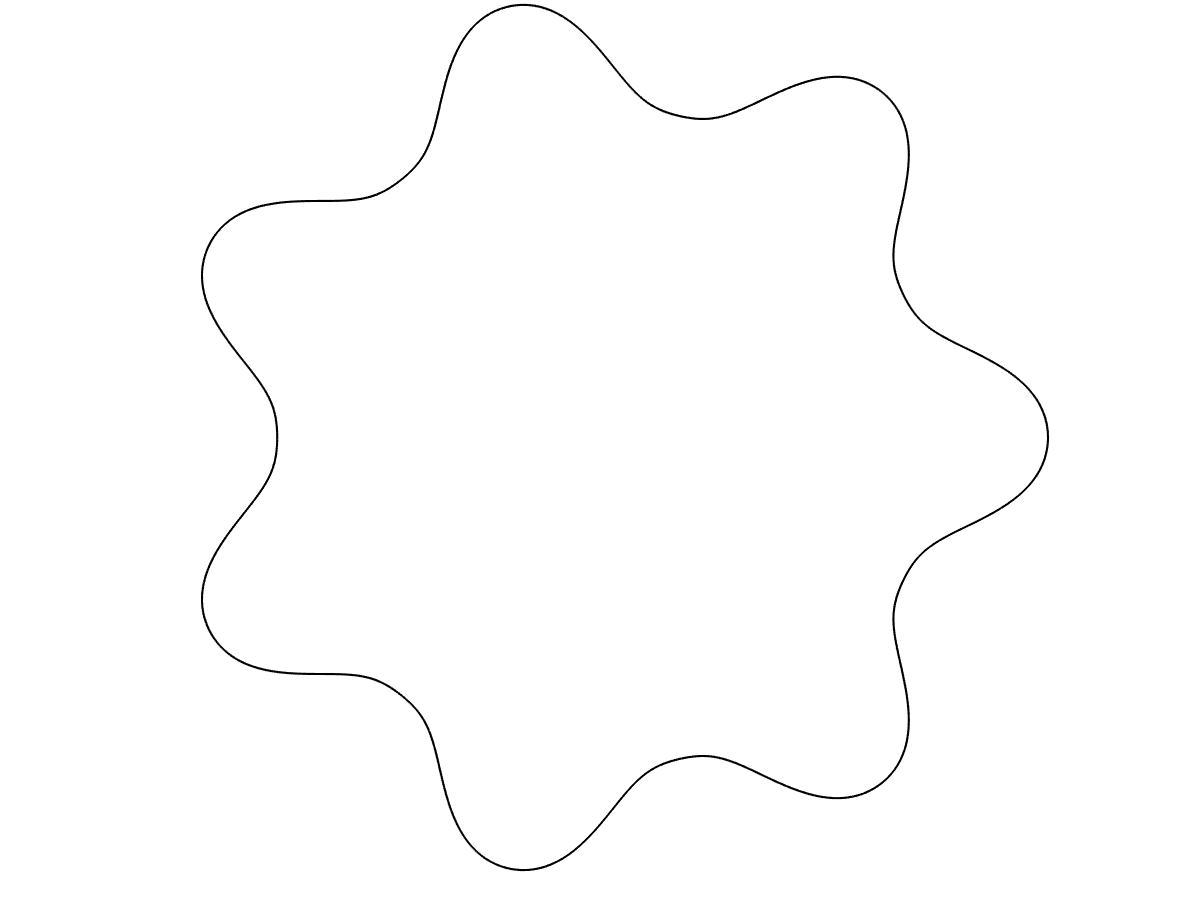}
\includegraphics[width=.3\linewidth]{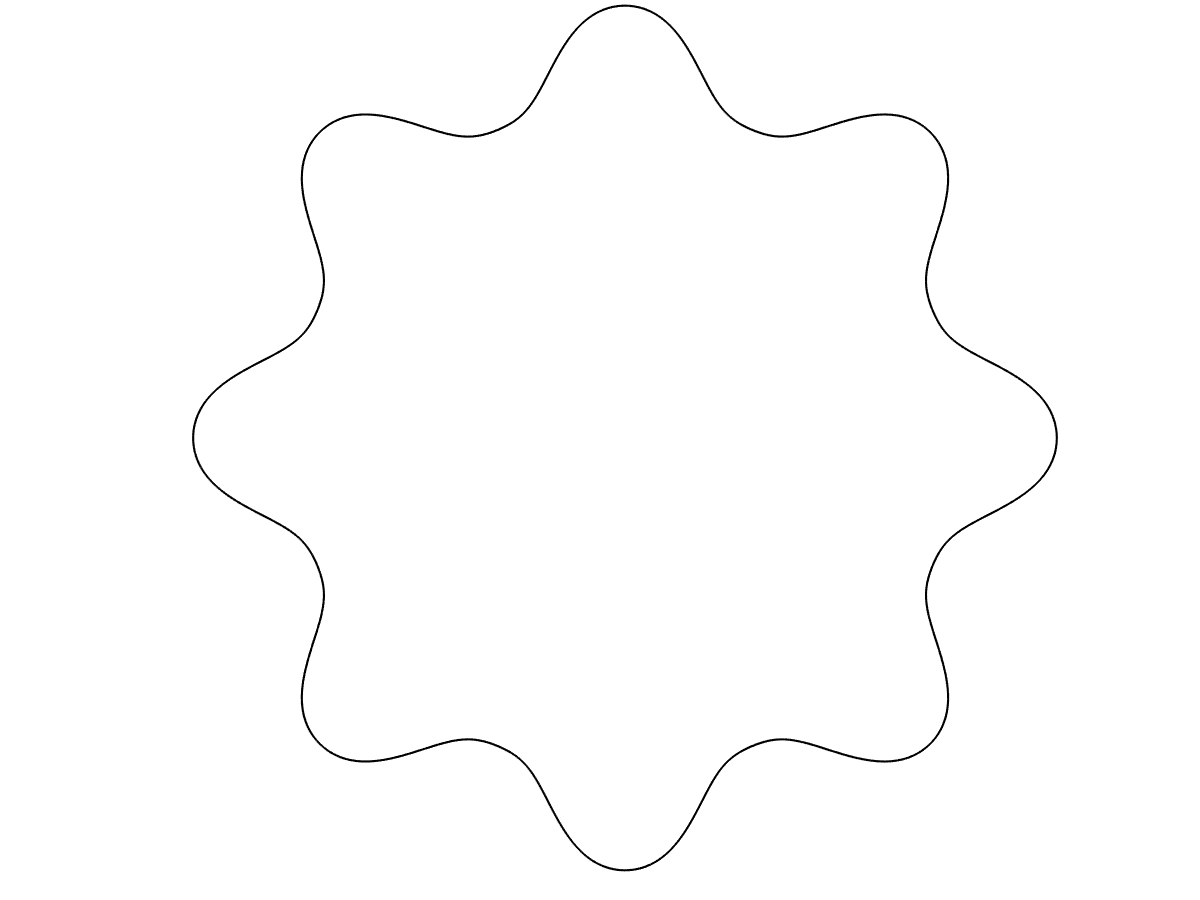}
\includegraphics[width=.3\linewidth]{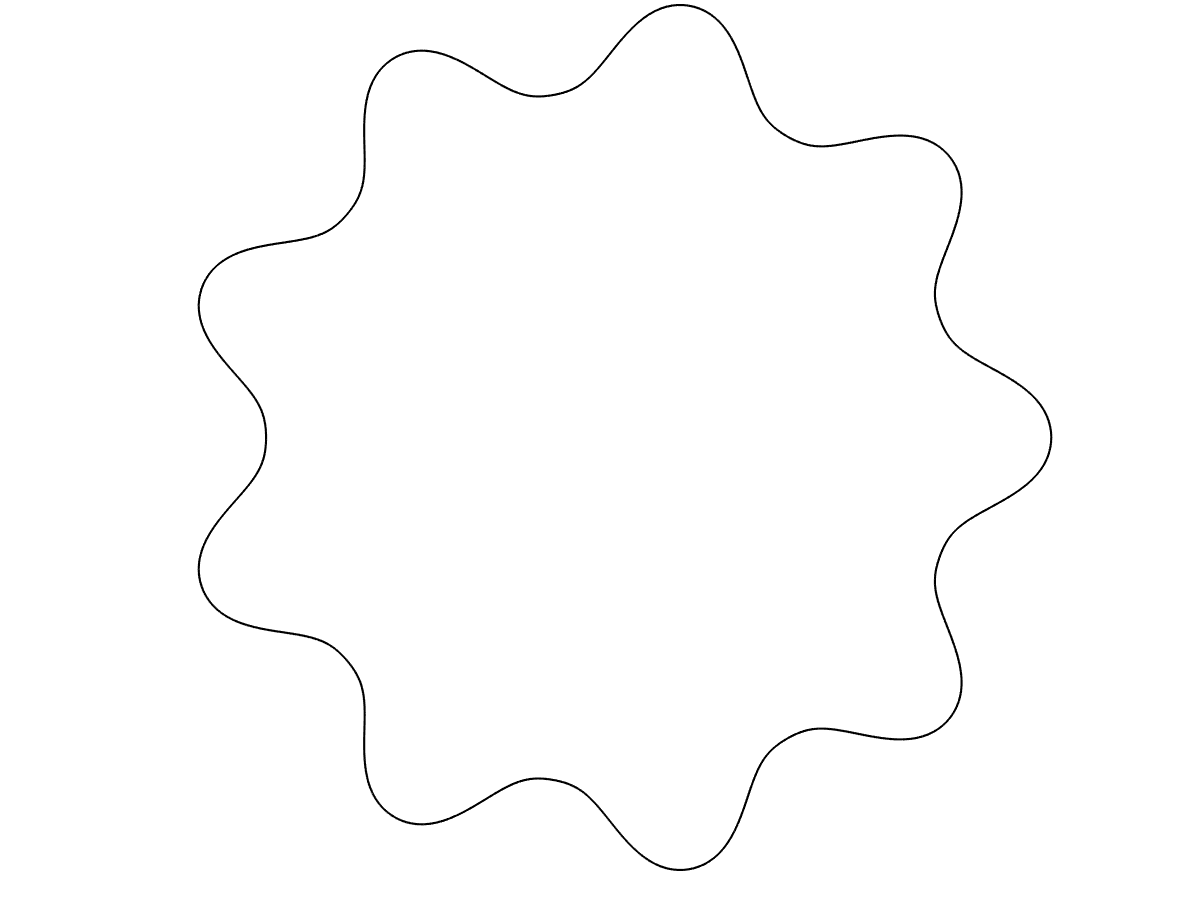}
\includegraphics[width=.3\linewidth]{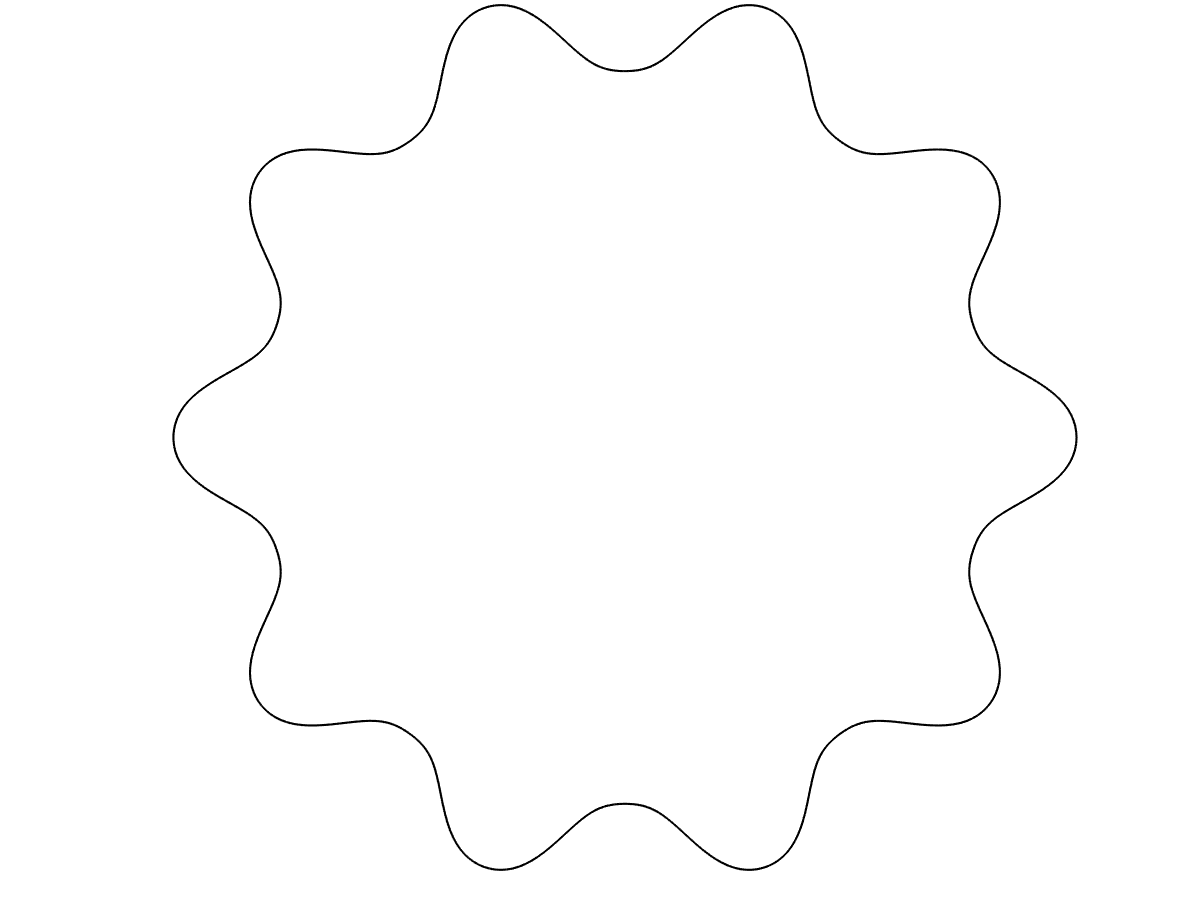}

%\vspace{.5cm}

{\scriptsize 
\begin{tabular}{r |r r r r r}
 j/p & 1   & 2   & 3   & 4   & 5   \\ 
 \hline 
1 & \bf 1.77245385087 & 0.77698864096 & 1.07942827817 & 1.17095593776 & 1.23886424463 \\ 
2 & \bf 1.77245385087 & \bf 2.91496429809 & 1.07942827817 & 1.17095593776 & 1.23886424463 \\ 
3 & 3.54489505800 & \bf 2.91496429809 & \bf 4.14395657280 & 1.61092851928 & 1.94478915686 \\ 
4 & 3.54492034560 & 3.28021642525 & \bf 4.14395657280 & \bf 5.28230087347 & 1.94478915686 \\ 
5 & 5.31736077095 & 4.45733853748 & \bf 4.14395657280 & \bf 5.28230087347 & \bf 6.49379637444 \\ 
6 & 5.31736233451 & 5.07531707453 & 4.91719025908 & 5.44583490769 & \bf 6.49379637444 \\ 
7 & 7.08981538184 & 6.12892820432 & 6.01189780569 & 5.44583490769 & \bf 6.49379637444 \\ 
8 & 7.08981542529 & 6.24787066243 & 6.01189780569 & 6.49483453114 & 6.72879743688 \\ 
9 & 8.86226925401 & 7.72388611906 & 7.63597621842 & 7.32884878844 & 6.72879743688 \\ 
10 & 8.86226925490 & 7.78555184785 & 7.63597621842 & 7.32884878844 & 8.13434132277 \\ 
11 & 10.63472310534 & 9.20742986631 & 8.93164119659 & 8.54954592721 & 8.80070046267 \\ 
12 & 10.63472310535 & 9.35623119578 & 9.13933506157 & 9.11731770841 & 8.80070046267 \\ 
 &  &  &  &  &  \\ 
  j/p & 6   & 7   & 8   & 9   & 10   \\ 
 \hline 
1 & 1.26563770224 & 1.29215311002 & 1.30399980096 & 1.31769945903 & 1.32419993715 \\ 
2 & 1.26563770224 & 1.29215311002 & 1.30399980096 & 1.31769945903 & 1.32419993715 \\ 
3 & 2.11876408010 & 2.25268514632 & 2.33026127056 & 2.39855173093 & 2.44116583782 \\ 
4 & 2.11876408010 & 2.25268514632 & 2.33026127056 & 2.39855173093 & 2.44116583782 \\ 
5 & 2.42888722971 & 2.78130732355 & 3.00533497492 & 3.18580541676 & 3.30698927275 \\ 
6 & \bf 7.64164323380 & 2.78130732355 & 3.00533497492 & 3.18580541676 & 3.30698927275 \\ 
7 & \bf 7.64164323381 & \bf 8.84377279901 & 3.24435955530 & 3.60805068461 & 3.86493960791 \\ 
8 & 7.76830589795 & \bf 8.84377279901 & \bf 9.99777577159 & 3.60805068461 & 3.86493960791 \\ 
9 & 7.76830589795 & \bf 8.84377279901 & \bf 9.99777577159 & \bf 11.19446201555 & 4.05917248381 \\ 
10 & 7.97457881941 & 9.04746859718 & 10.09825512148 & \bf 11.19446201555 & \bf 12.35253261747 \\ 
11 & 7.97457881941 & 9.04746859718 & 10.09825512148 & \bf 11.19446201555 & \bf 12.35253261747 \\ 
12 & 9.73477342826 & 9.22062775990 & 10.32313009868 & 11.36535997845 & 12.43514930841
\end{tabular}}
\caption{{\bf (top)} $\Omega^{p\star}$ for $p=2\ldots 10$. The optimal domain for $p=1$ is a ball. {\bf (bottom)}  Values $\Lambda_j(\Omega^{p\star})$ for $p=1,\ldots,10$ and $j=1,\ldots,12$. 
%Note that  $\Lambda_p(\Omega^{p\star})$ and $\Lambda_{p+1}(\Omega^{p\star})$ agree to twelve digits. 
See \S\ref{sec:NumRes}.}
\label{fig:Ap}
\end{center}
\end{figure}

\begin{figure}[t!]
\begin{center}
\includegraphics[width=1\linewidth,trim={0.2cm 5.5cm 0.2cm 5.5cm},clip]{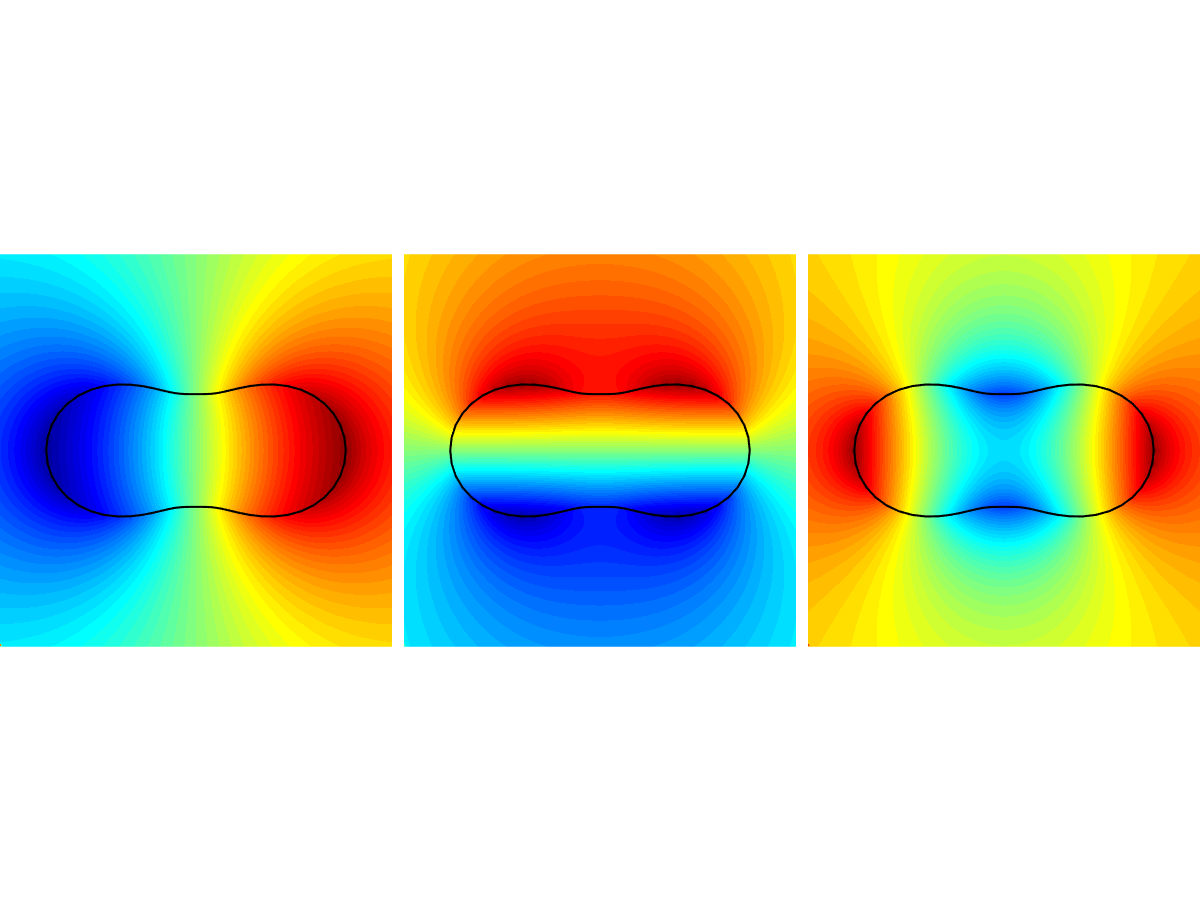} 
\vspace{.1cm}

\includegraphics[width=1\linewidth,trim={0.2cm 5.5cm 0.2cm 5.5cm},clip]{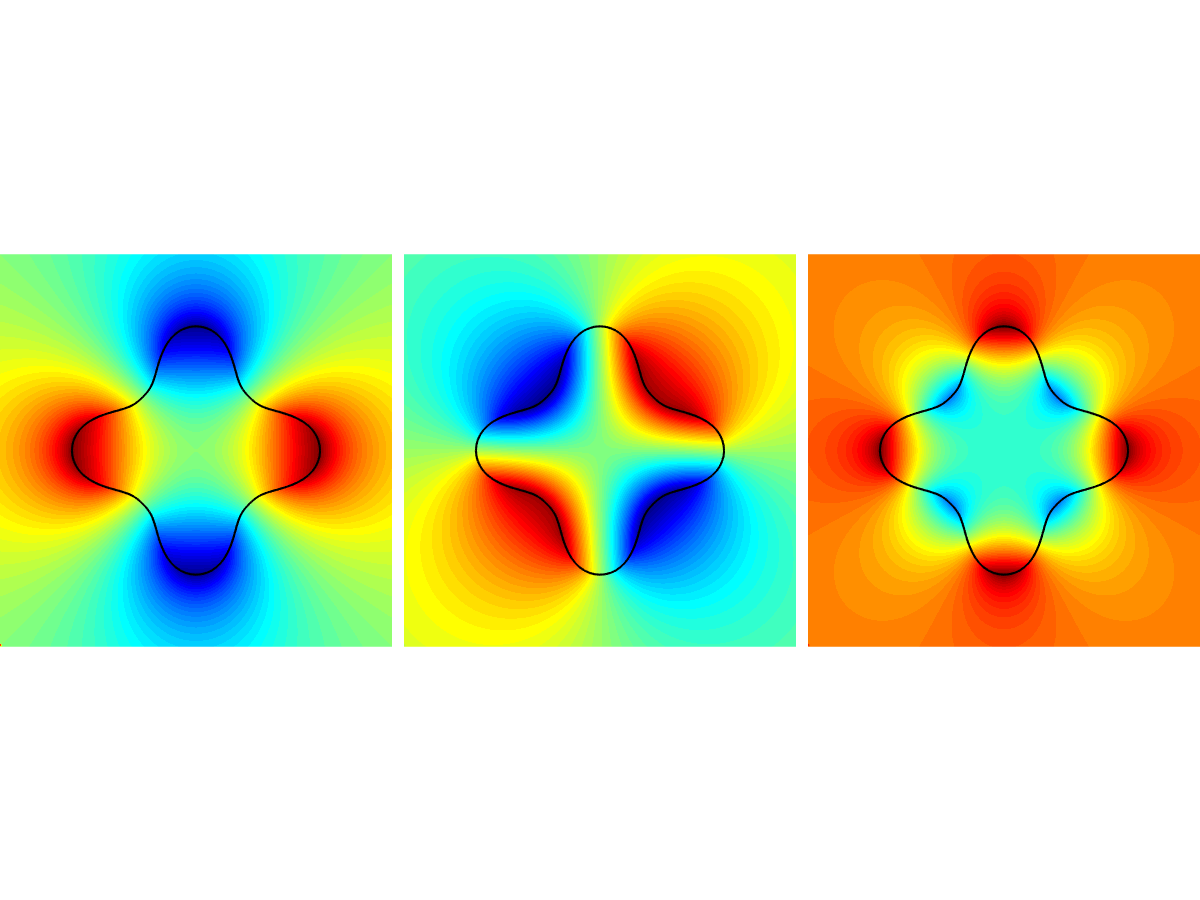} 
\vspace{.1cm}

\includegraphics[width=1\linewidth,trim={0.2cm 5.5cm 0.2cm 5.5cm},clip]{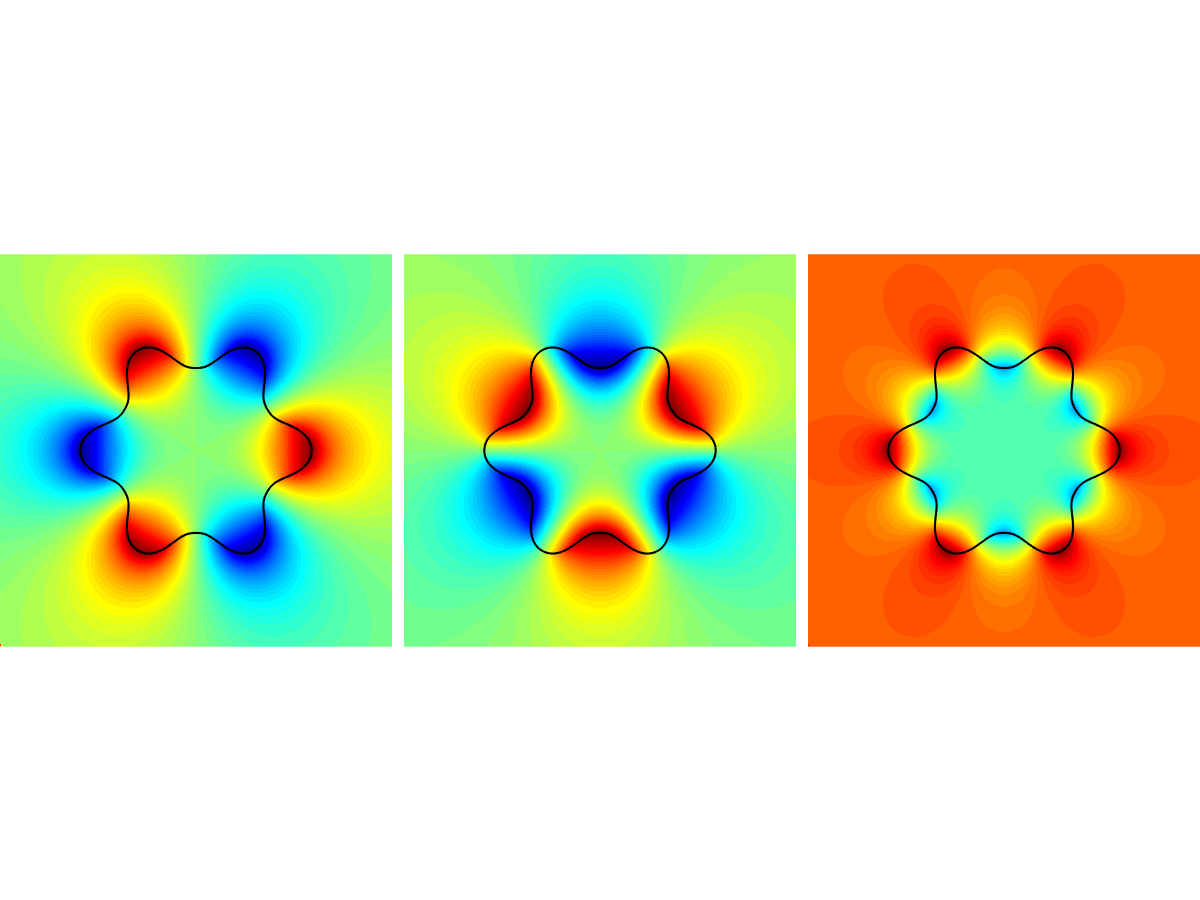} 
\vspace{.1cm}

\includegraphics[width=1\linewidth,trim={0.2cm 5.5cm 0.2cm 5.5cm},clip]{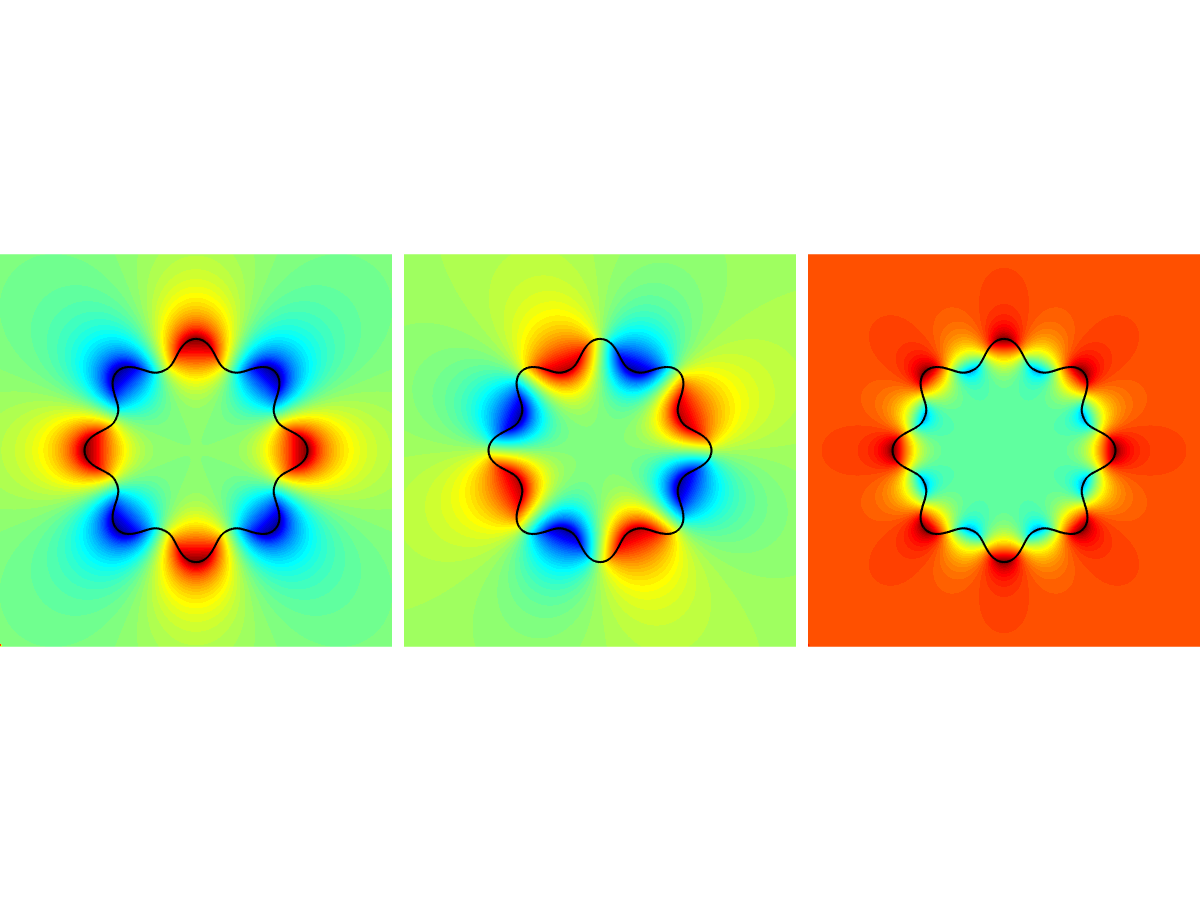} 
\vspace{.1cm}

\includegraphics[width=1\linewidth,trim={0.2cm 5.5cm 0.2cm 5.5cm},clip]{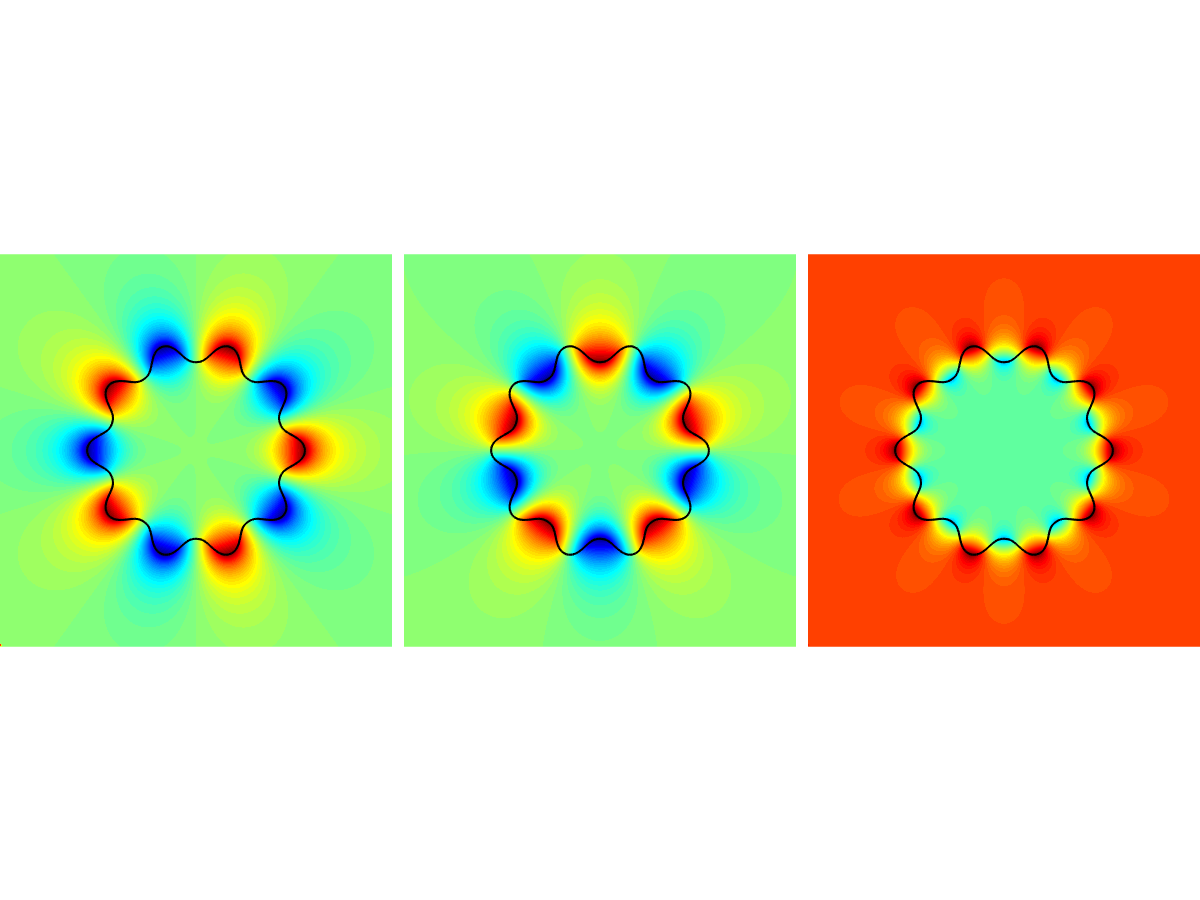}

\caption{For $\Omega^{p\star}$ with even $p=2, 4, 6, 8, 10$, Steklov eigenfunctions $p-1$, $p$, and $p+1$. 
Here, $\Lambda_{p-1} < \Lambda_p = \Lambda_{p+1} < \Lambda_{p+2}$. 
%See \S\ref{sec:NumRes}. 
}
\label{fig:EigFunsEven}
\end{center}
\end{figure}

\begin{figure}[t!]
\begin{center}
\includegraphics[width=1\linewidth,trim={0.2cm 5.5cm 0.2cm 5.5cm},clip]{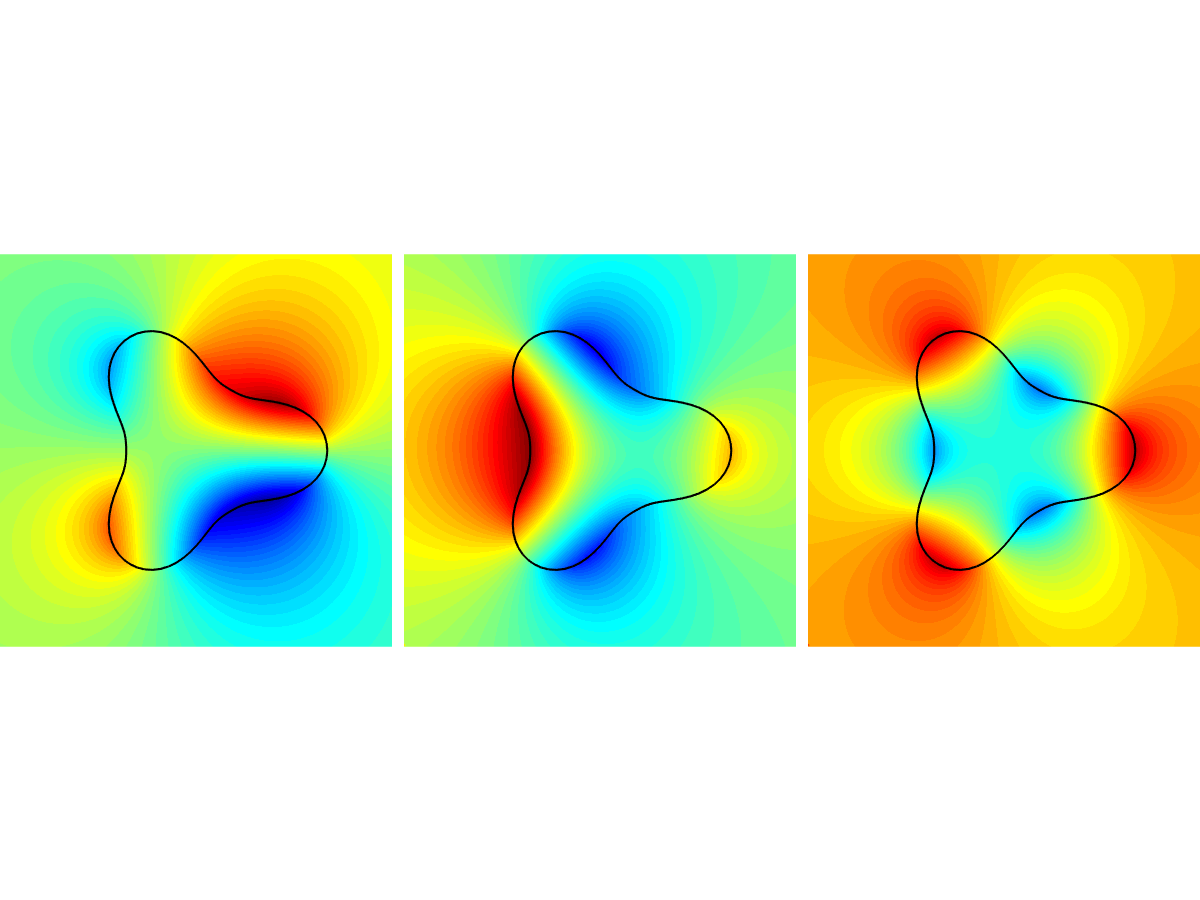} 
\vspace{.1cm}

\includegraphics[width=1\linewidth,trim={0.2cm 5.5cm 0.2cm 5.5cm},clip]{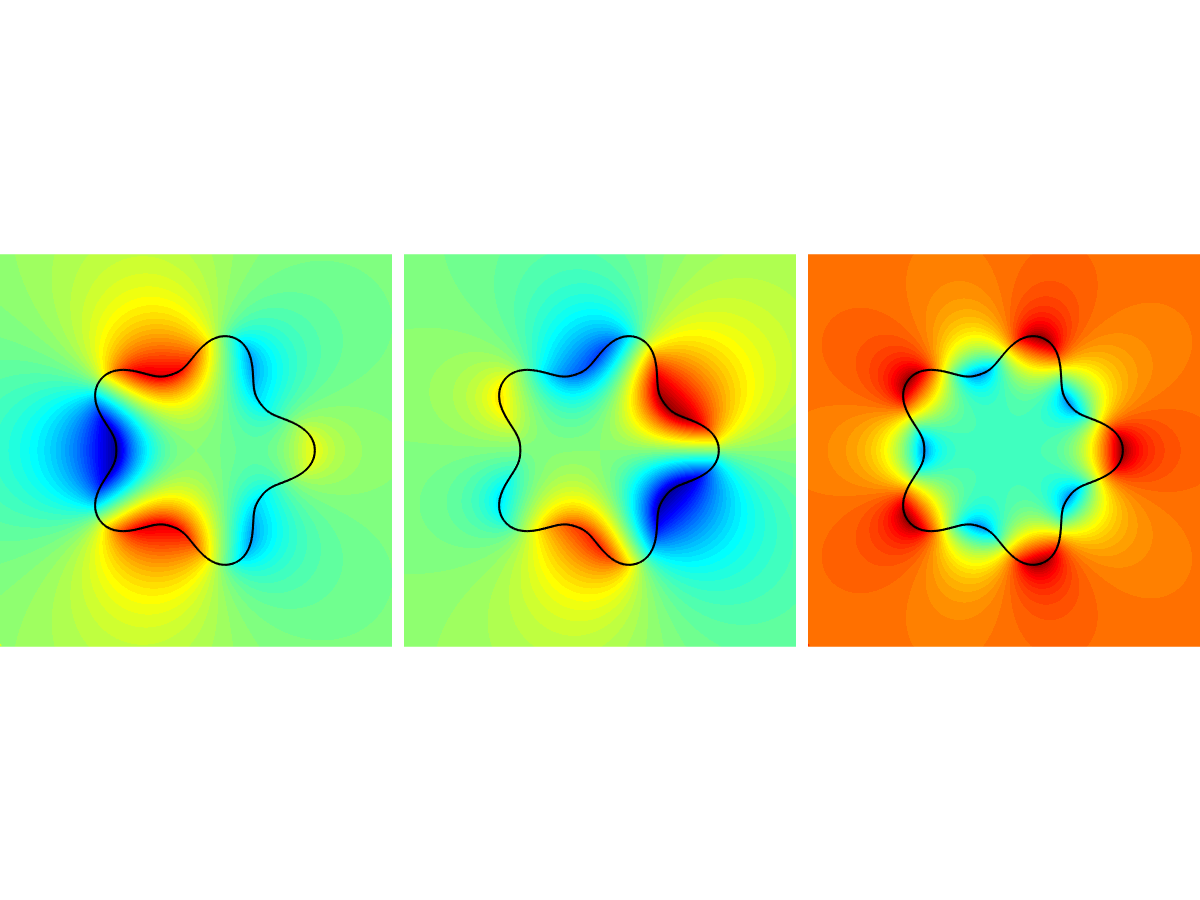} 
\vspace{.1cm}

\includegraphics[width=1\linewidth,trim={0.2cm 5.5cm 0.2cm 5.5cm},clip]{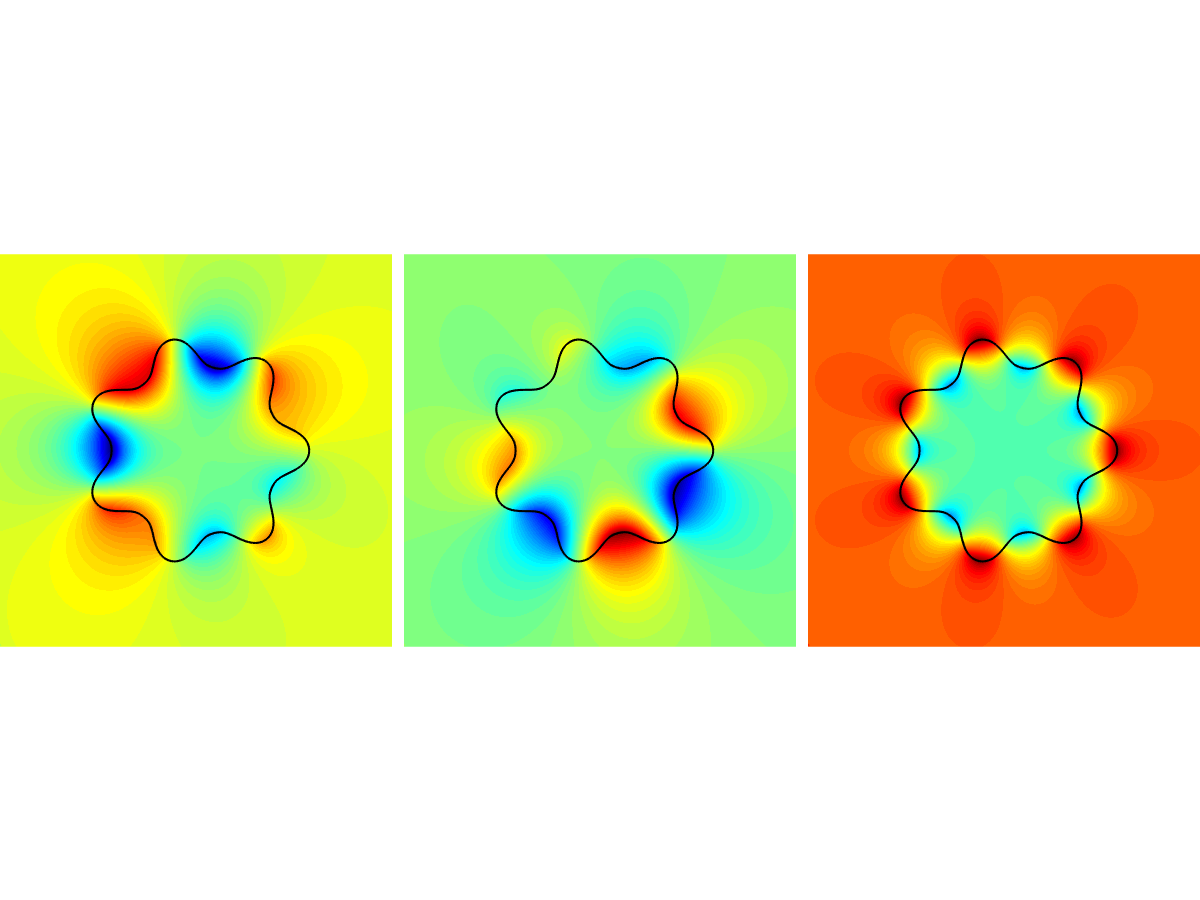} 
\vspace{.1cm}

\includegraphics[width=1\linewidth,trim={0.2cm 5.5cm 0.2cm 5.5cm},clip]{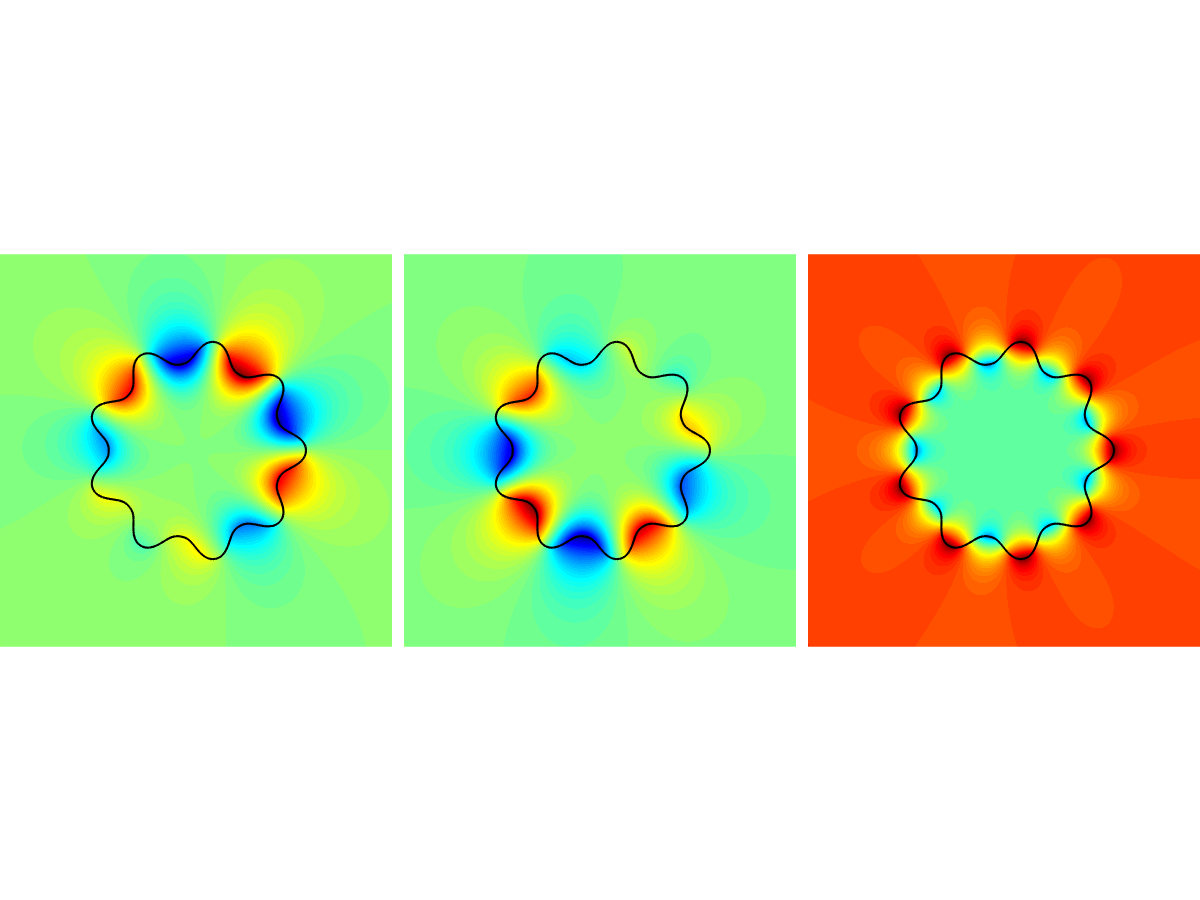} 

\caption{For $\Omega^{p\star}$ with odd $p=3, 5,7, 9$, Steklov eigenfunctions $p$, $p+1$, and  $p+2$. 
Here, $\Lambda_{p-1} < \Lambda_p = \Lambda_{p+1} = \Lambda_{p+2} < \Lambda_{p+3}$.  
%See \S\ref{sec:NumRes}.
}
\label{fig:EigFunsOdd}
\end{center}
\end{figure}

\subsubsection*{Initial Results.} Optimal domains, $\Omega^{p\star}$, for $\Lambda^{p\star}$ for $p=2\ldots 10$ are plotted in Figure~\ref{fig:Ap}. We also define $\Lambda_j^{p\star} := \Lambda_j( \Omega^{p\star})$ and tabulate $\Lambda_j^{p\star}$ for $j=1,\ldots 12$. 
In Figures \ref{fig:EigFunsEven} and \ref{fig:EigFunsOdd}, we plot 
the eigenfunctions corresponding to $\Lambda_j^{p\star} $ for $ j = p-1, \  p, \ p+1$ if $p$ is even and $j = p, \ p+1, \ p+2$ if $p\geq3$ is odd. 
The eigenfunctions are extended outside of $\Omega^{p\star}$ using the representation~\eqref{steklov_single_layer_modified}.
The optimal domains and their  eigenpairs are very structured. 
Namely, for these values of $p$, we make the following (numerical) observations:
\begin{enumerate}
\item The optimal domains, $\Omega^{p\star}$, are unique (up to dilations and rigid transformations).  
\item $\Omega^{p\star}$ looks like a ``ruffled pie dish''  with $p$ ``ruffles'' where the curvature of the boundary is positive. In particular, $\Omega^{p\star}$ has $p$-fold rotational symmetry and an axis of symmetry.
\item The $p$-th eigenvalue has multiplicity 2 for $p$ even and multiplicity 3 for $p\geq 3$ odd, {\it i.e.}, 
\begin{align*}
\text{$p$ even: } \quad & \Lambda_p^{p\star} = \Lambda_{p+1}^{p\star} <\Lambda_{p+2}^{p\star} \\
\text{$p$ odd: } \quad & \Lambda_p^{p\star} = \Lambda_{p+1}^{p\star} = \Lambda_{p+2}^{p\star} <  \Lambda_{p+3}^{p\star}.
\end{align*}
\item There is a very large gap between $\Lambda_{p-1}^{p\star}$ and $\Lambda_p^{p\star}$. For even $p$, $\Lambda_{p-1}^{p\star}$ is simple.
\item For even $p$, the  eigenfunction  corresponding to $\Lambda_{p-1}$ (left) and two eigenfunctions from the eigenspace corresponding to $\Lambda_p = \Lambda_{p+1}$ (center and right) are plotted in Figure  \ref{fig:EigFunsEven}.   The eigenfunctions are all nearly zero at the center of the domain and oscillatory on the boundary. The  eigenfunction  corresponding to $\Lambda_{p-1}$ takes alternating maxima and minima on the ``ruffles'' of the domain.  Eigenfunctions from the $\Lambda_p = \Lambda_{p+1}$  eigenspace may be chosen  so that one eigenfunction is nearly zero on the ``ruffles'' of the domain and takes alternating maxima and minima in-between. The other eigenfunction takes maxima on the ``ruffles'' of the domain and minima in-between.

For odd $p\geq3$, in Figure  \ref{fig:EigFunsOdd}, we plot the eigenfunctions from the three-dimensional eigenspace corresponding to $ \Lambda_p = \Lambda_{p+1} = \Lambda_{p+2} $. 
Again, eigenfunctions from this subspace are nearly zero on the interior of the domain and oscillatory on the boundary. They may be chosen so that, again, one eigenfunction takes maxima on the ``ruffles'' of the domain and minima in-between (right figures). The other two eigenfunctions are nearly zero on the  ``ruffles'' of the domain and take alternating maxima and minima in-between on the boundary. These two eigenfunctions are concentrated on opposite sides of the domain. 
\end{enumerate}

Some of these observations are also summarized in Conjecture \ref{optDom}. For $p=2,3,4,5$, the domain symmetries  can also been observed in the recent numerical results of B. Bogosel \cite{Bogosel2015}. 
Preliminary results indicate that the introduction of a hole in the domain decreases the $p$-th eigenvalue. To consider larger values of $p$, we use the structure of the optimal domains for relatively small $p$ to  reduce the search space and generate good initial domains for the optimization procedure. 

\subsubsection*{Structured Coefficients.}
If a domain has $p$-fold symmetry, the only non-zero coefficients in the Fourier expansion \eqref{eq:BndyExpan} are multiples of $p$.  If there is an axis of symmetry, then we can assume  $b_k = 0$ for $k\geq 1$. 
Therefore, when minimizing the $p$-th eigenvalue, we only vary the coefficients
$a_{k\cdot p}$ for $k=1,2,3$.
This simplification reduces the shape optimization problem to an optimization  problem with  just 3 parameters.

Let $\{a_j^{p\star}\}_{j}$ denote the coefficients corresponding to $\Omega^{p\star}$. 
Solving  \eqref{eq:Ap} for $p\leq 40$, we observe that, as a function of $p$, the coefficients $\{a^{p\star}_{k\cdot p}\}_{k=1,2,3}$ decay at a rate $a^{p\star}_{k\cdot p} \propto \frac{1}{p}$. Using computed values for the optimal coefficients, we obtain the following interpolations, denoted $\{ a^{p\circ}_j \}_j$,
\begin{equation} \label{eq:Interp}
a^{p\circ}_{1\cdot p} = \frac{1}{0.1815 + 0.3444 \cdot p} \quad 
a^{p\circ}_{2\cdot p} = \frac{1}{-6.1198 + 7.6443 \cdot p} \quad 
a^{p\circ}_{3\cdot p} = \frac{1}{-4.5563 -7.6561\cdot p}. 
\end{equation}
A plot of these three interpolations is given in Figure \ref{fig:OptCoeffs}(left). Let $\Omega^{p\circ}$ denote the domain corresponding to these coefficients, $\{ a^{p\circ}_j \}_j$. 
In  Figure \ref{fig:OptCoeffs}(center), we plot $\Lambda_p(\Omega^{p\circ})$ in blue, $\Lambda_p(\Omega^{p\star})$ in red, and the value of $\Lambda_p$ for a ball in black. The values for $\Omega^{p\circ}$ and $\Omega^{p\star}$ are indistinguishable, although the multiplicity of the $p$-th eigenvalue for these two domains differs. We observe that the value of $\Lambda_p(\Omega^{p \circ})$ grows linearly with $p$. Linear interpolation of $\Lambda_p(\Omega^{p\circ})$ gives  
\begin{equation} \label{eq:InterpOptVals}
\Lambda_p(\Omega^{p\circ}) \approx 0.5801 + 1.1765 \cdot p. 
\end{equation}
Linear interpolation of the eigenvalues of a ball gives
$$
\Lambda_p(\mathbb B) \approx 0.4436 + 0.8862 \cdot p. 
$$
The interpolation for a ball is in good agreement with Weyl's law, $\lambda_j(\Omega) |\partial \Omega| \sim j \pi$,  
since for a ball we have 
$| \partial \Omega | = 2 \sqrt{\pi} | \Omega |^{\frac{1}{2}} $ 
and 
$\Lambda_p(\mathbb B^2) = \lambda_p(\mathbb B^2) \sqrt{| \mathbb B^{2} |} = \frac{1}{2 \sqrt{\pi} } \lambda_p(\mathbb B^2) | \partial \mathbb B^{2} | \sim \frac{\sqrt{\pi}}{2} \cdot p $. One can view \eqref{eq:InterpOptVals} in terms of the bound given in \eqref{eq:LinearGrowth}. In dimension two, using the isoperimetric inequality, $4 \pi |\Omega| \leq  |\partial \Omega |^2$, we have that 
$$
\Lambda_p(\Omega) = \lambda_p(\Omega) \cdot |\Omega|^{\frac{1}{2}} \leq \frac{1}{2 \sqrt{\pi}}  \lambda_p(\Omega)\cdot |\partial \Omega| \leq  \tilde{C} p .
$$
We have constructed a sequence of domains with maximal value $\Lambda_p(\Omega)$, so have computed the (optimal) value of $\tilde C$ in this inequality.  For the  interpolated domains, $\Omega^{p\circ}$, we plot $p$ vs. the perimeter, $| \partial \Omega^{p\circ}| / \sqrt{ |  \Omega^{p\circ}| }$,   in Figure \ref{fig:OptCoeffs}(right). We observe that the perimeter appears to converge to a value near $4.53$, which is greater than the value for the disc, $ 2 \sqrt{\pi} \approx 3.54$.

\subsubsection*{Solution Of \eqref{eq:Ap} For Large $p$.} We extrapolate the interpolation given in  \eqref{eq:Interp} to $p=50, 51, 100, 101$. Using this as an initial condition for the optimization problem \eqref{eq:Ap} where we restrict the admissible set to domains with coefficients $ a_{k\cdot p}$ for $k=1,2,3$, we solve the optimization problem to obtain domains, $\Omega^{p\star}$, plotted in Figure \ref{fig:highP}.  Here, we also tabulate $\Lambda_j^{p\star}$ for $j=p-2,\ldots p+4$. In Figures \ref{fig:eigFunp50100} and  \ref{fig:eigFunp51101}, we plot the eigenfunctions corresponding to $\Lambda_j^{p\star} $ for $ j = p-1, \  p, \ p+1$ if $p$ is even and $j = p, \ p+1, \ p+2$ if $p$ is odd. The observations made above for small values of $p$ hold here as well.

\begin{figure}[t!]
\begin{center}
\includegraphics[width=.32\linewidth]{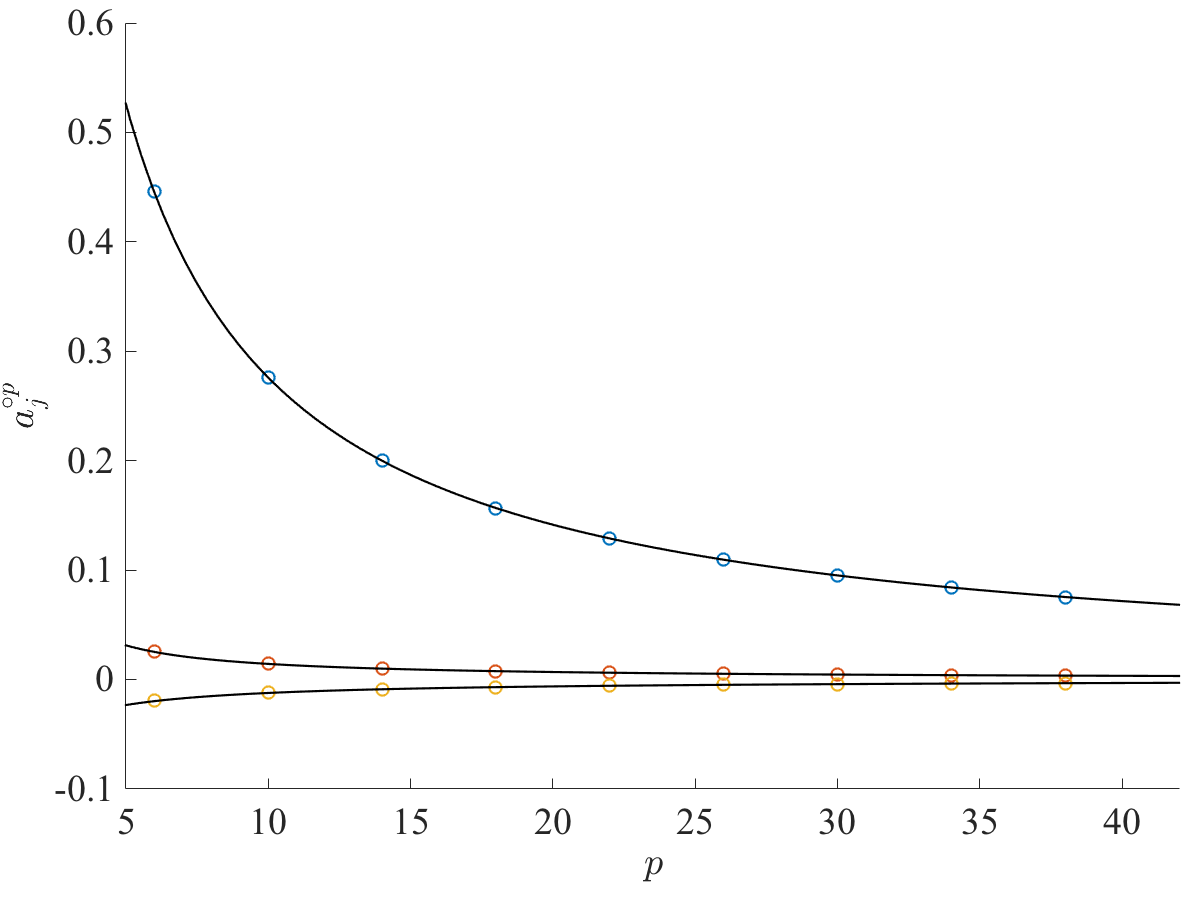}
\includegraphics[width=.32\linewidth]{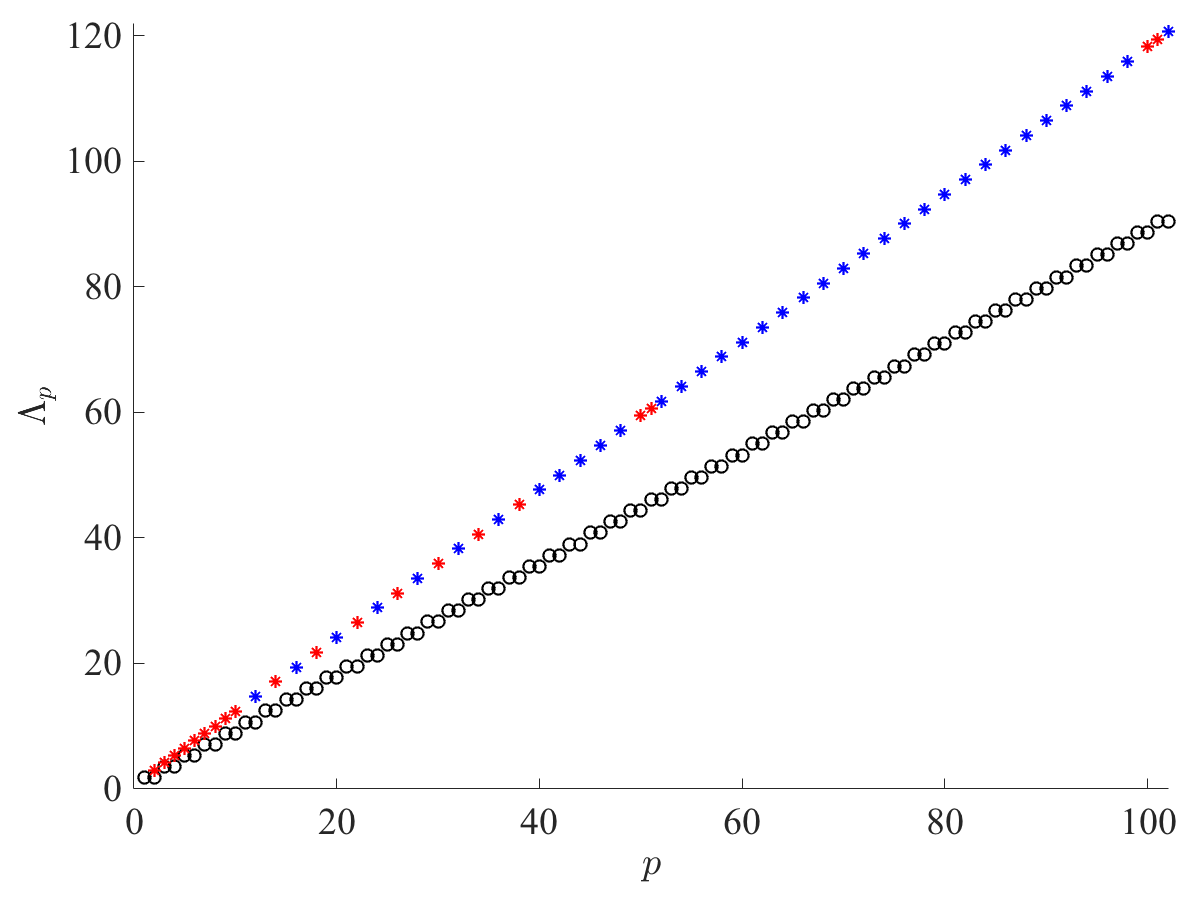}
\includegraphics[width=.32\linewidth]{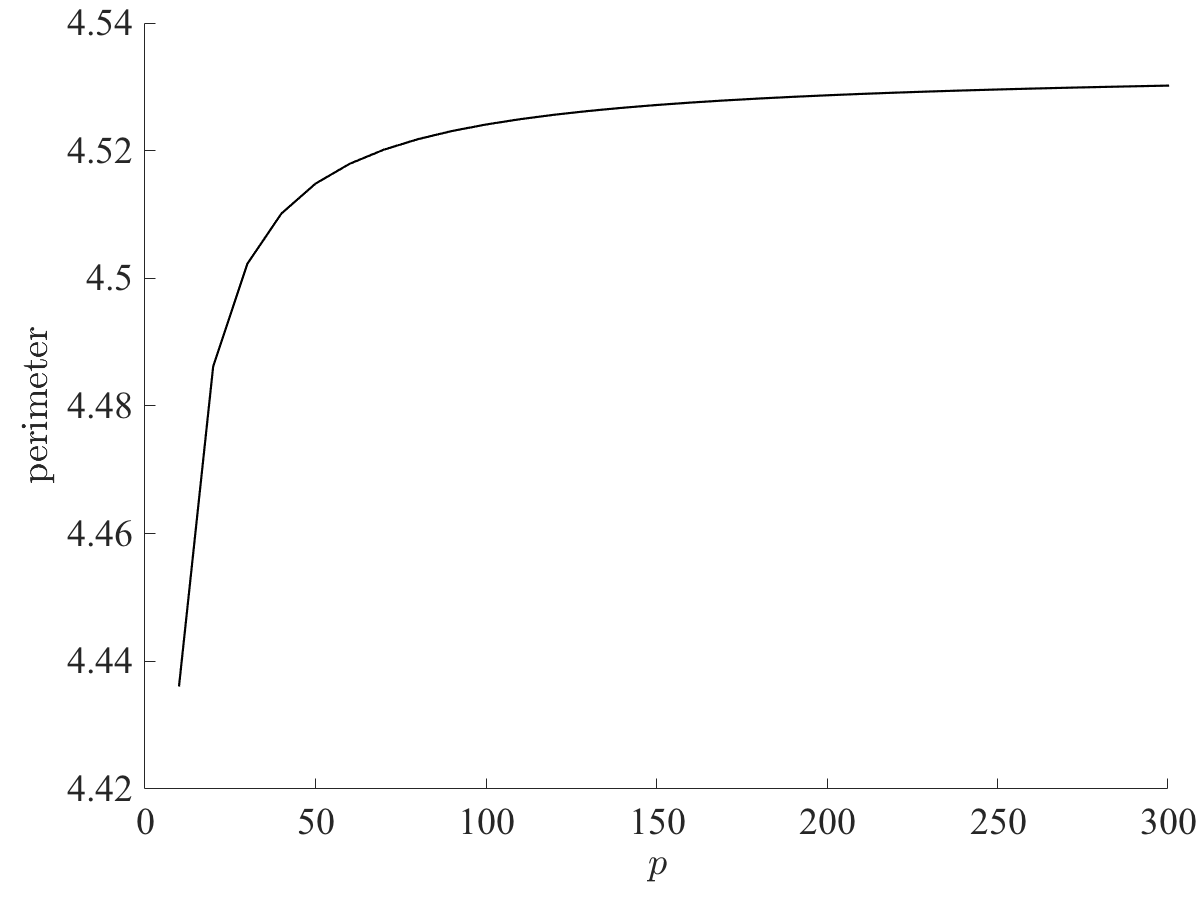}
\caption{{\bf (left)} For $p=6,10,14,\ldots 38$, a plot of the coefficients $a^{p\star}_{k\cdot p}$ for $k=1,2,3$ and the interpolation in  \eqref{eq:Interp}. 
{\bf (center)} Value of $\Lambda_p$ for a ball (black), 
interpolated domains $\Omega^{p\circ}$ (blue),  and 
optimal domains $\Omega^{p\star}$ (red). 
The values for $\Omega^{p\circ}$ and $\Omega^{p\star}$ are indistinguishable.
{\bf (right)} The perimeter of $\Omega^{p\circ}$ as a function of $p$. }   
\label{fig:OptCoeffs}
\end{center}
\end{figure}

\begin{figure}[t]
\begin{center}
\includegraphics[width=.45\linewidth]{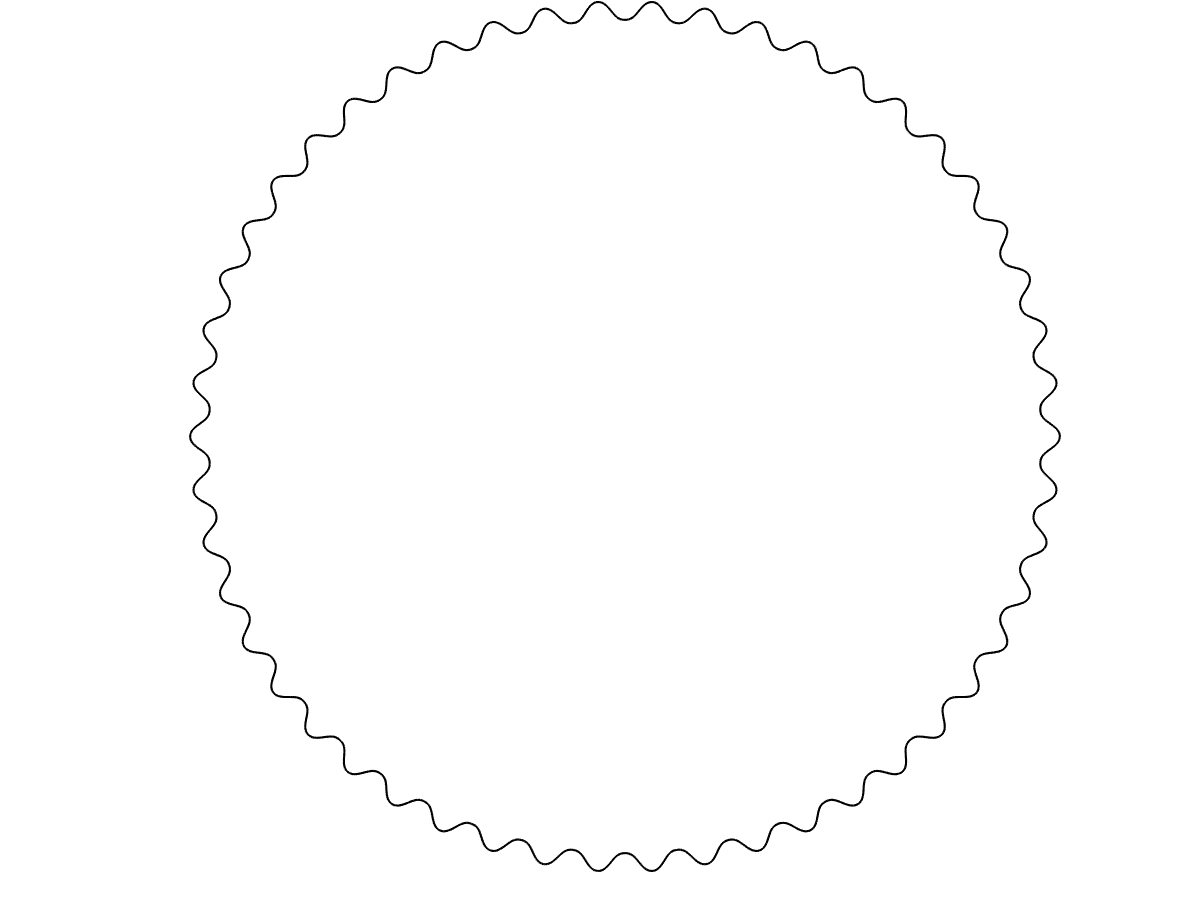} 
\includegraphics[width=.45\linewidth]{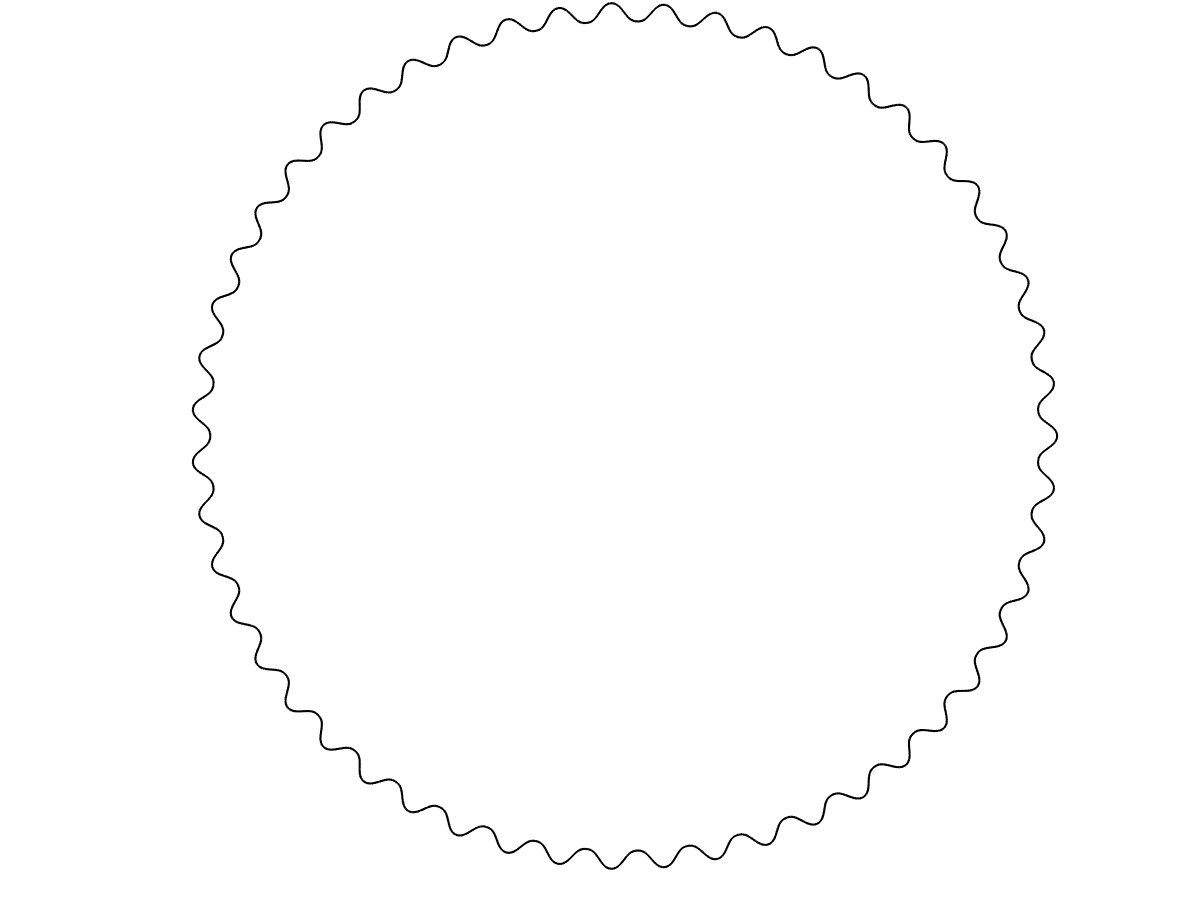} \\
\includegraphics[width=.45\linewidth]{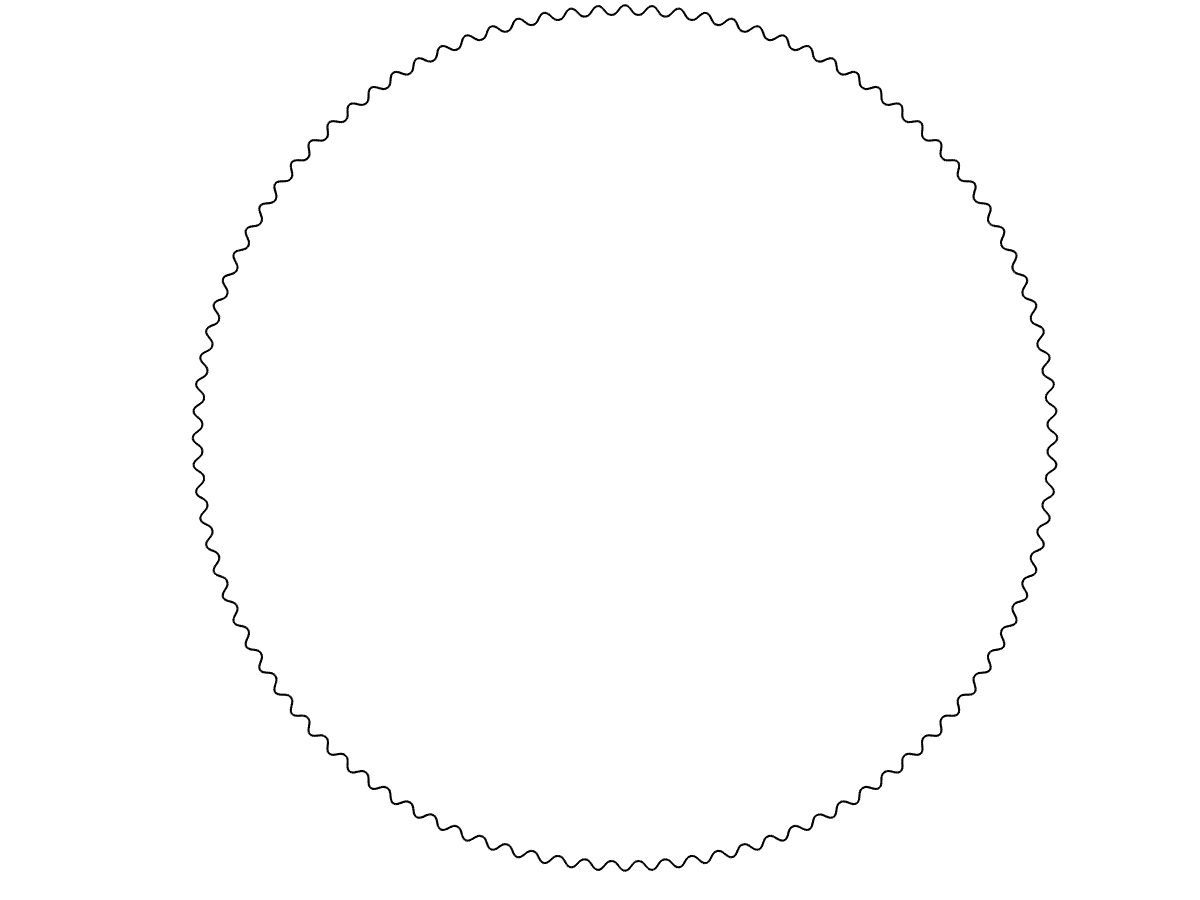}
\includegraphics[width=.45\linewidth]{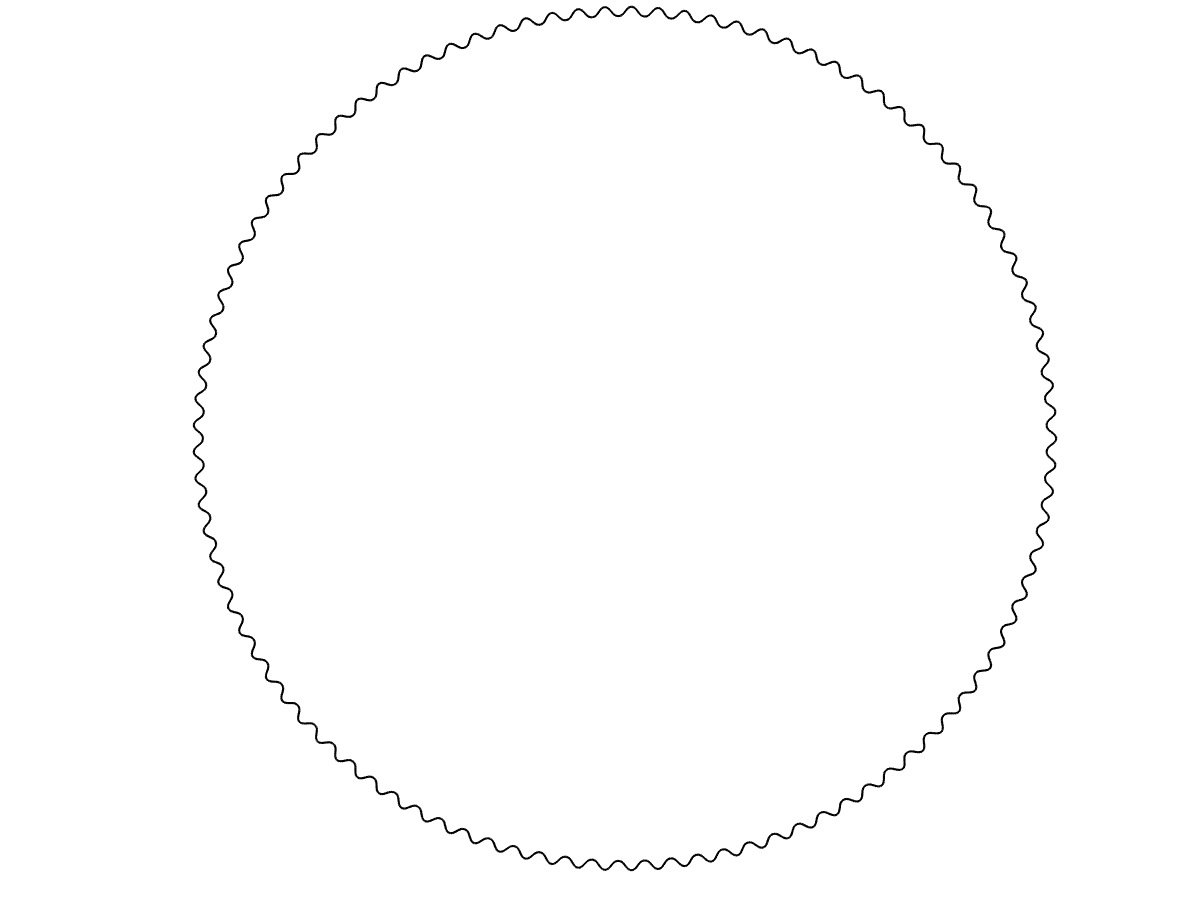}

\vspace{1cm}

{%\footnotesize 
\begin{tabular}{r |r r r r}
 j/p & 50   & 51   & 100   & 101   \\ 
 \hline 
$p-2$ & 20.30936391721 & 20.75355064378 & 40.69488836617 & 41.11849978908 \\ 
$p-1$ & 20.34992971509 & 20.75355064378 & 40.71526210762 & 41.11849978908 \\ 
$p$ & \bf 59.41361758262 & \bf 60.59374478101 & \bf 118.23330554334 & \bf 119.41159188027 \\ 
$p+1$ & \bf 59.41361758262 & \bf 60.59374478101 & \bf 118.23330554339 & \bf 119.41159188027 \\ 
$p+2$ & 59.43099705171 & \bf 60.59374478101 & 118.24200985153 & \bf 119.41159188027 \\ 
$p+3$ & 59.43099705171 & 60.62775851108 & 118.24200985153 & 119.42881860937 \\ 
$p+4$ & 59.48272776444 & 60.62775851108 & 118.26807069001 & 119.42881860937 
\end{tabular}}

\caption{{\bf (top)} $\Omega^{p\star}$ for $p=50, 51, 100, 101$.
{\bf (bottom)}  Values $\Lambda_j(\Omega^{p\star})$ for $j = p-2, \ldots, p+4$. 
%See \S\ref{sec:NumRes}.
}
\label{fig:highP}
\end{center}
\end{figure}

\begin{figure}[t]
\begin{center}
\includegraphics[width=1\linewidth,trim={0.2cm 5.5cm 0.2cm 5.5cm},clip]{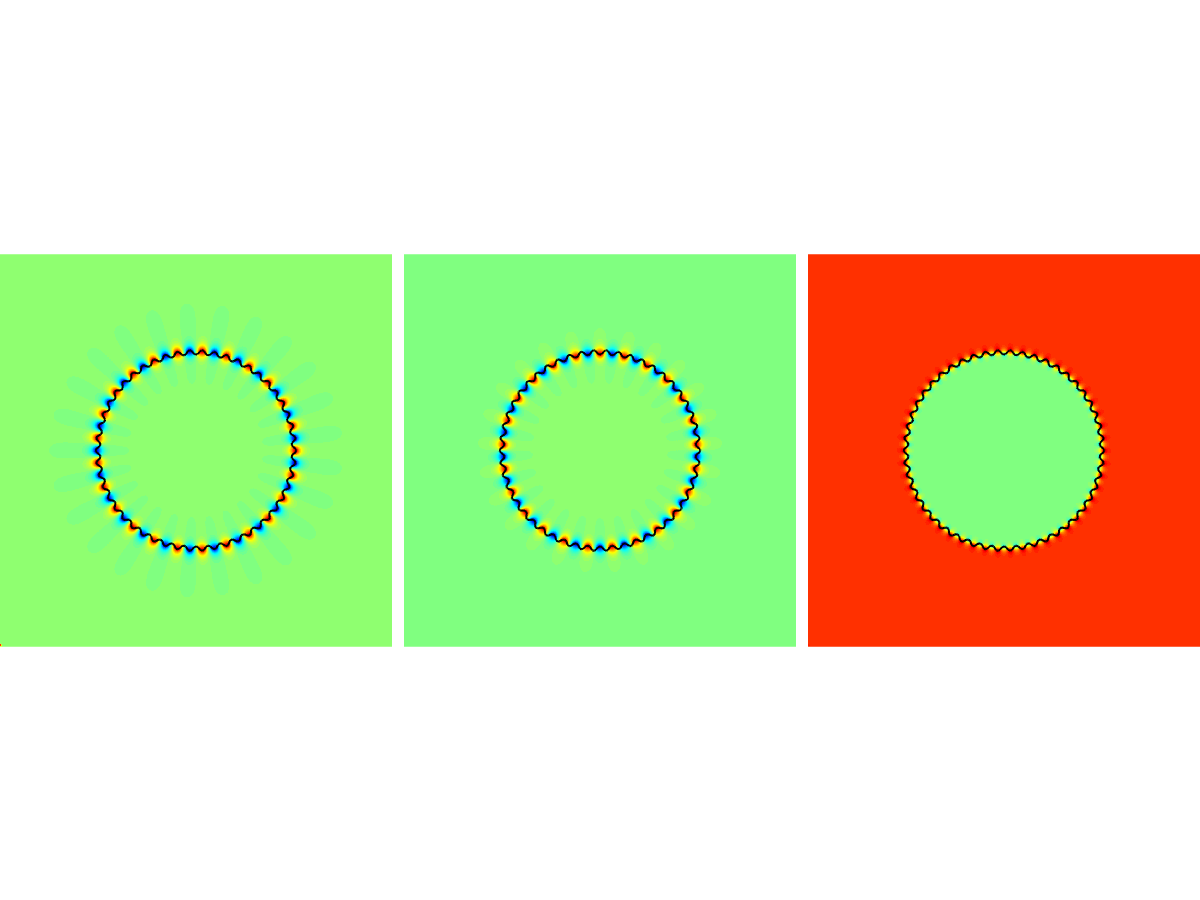} 
\vspace{.1cm}

\includegraphics[width=1\linewidth,trim={0.2cm 5.5cm 0.2cm 5.5cm},clip]{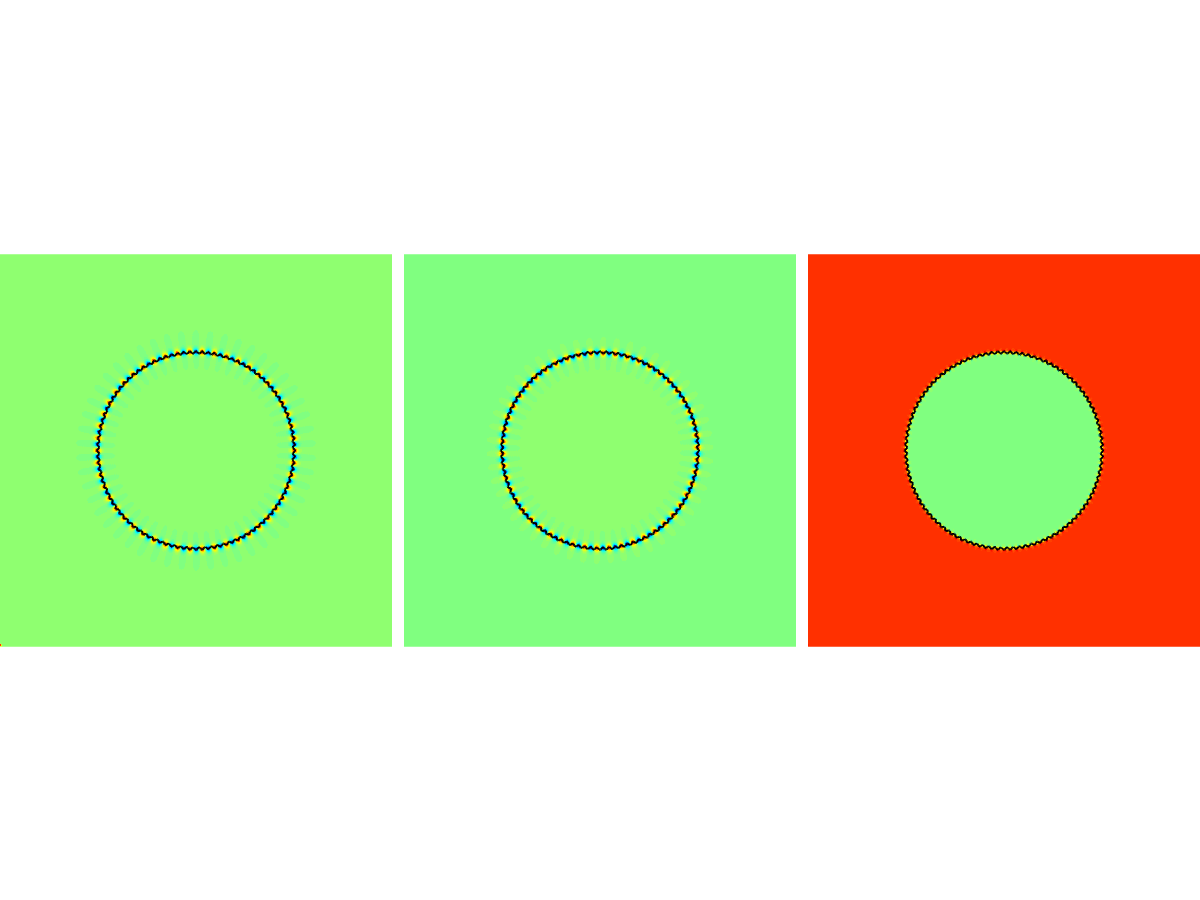} 
\caption{$p-1$, $p$, and $p+1$  Steklov eigenfunctions of $\Omega^{p\star}$   for $p=50$ (top) and $p=100$ (bottom).}
\label{fig:eigFunp50100}
\end{center}
\end{figure}

\begin{figure}[t]
\begin{center}
\includegraphics[width=1\linewidth,trim={0.2cm 5.5cm 0.2cm 5.5cm},clip]{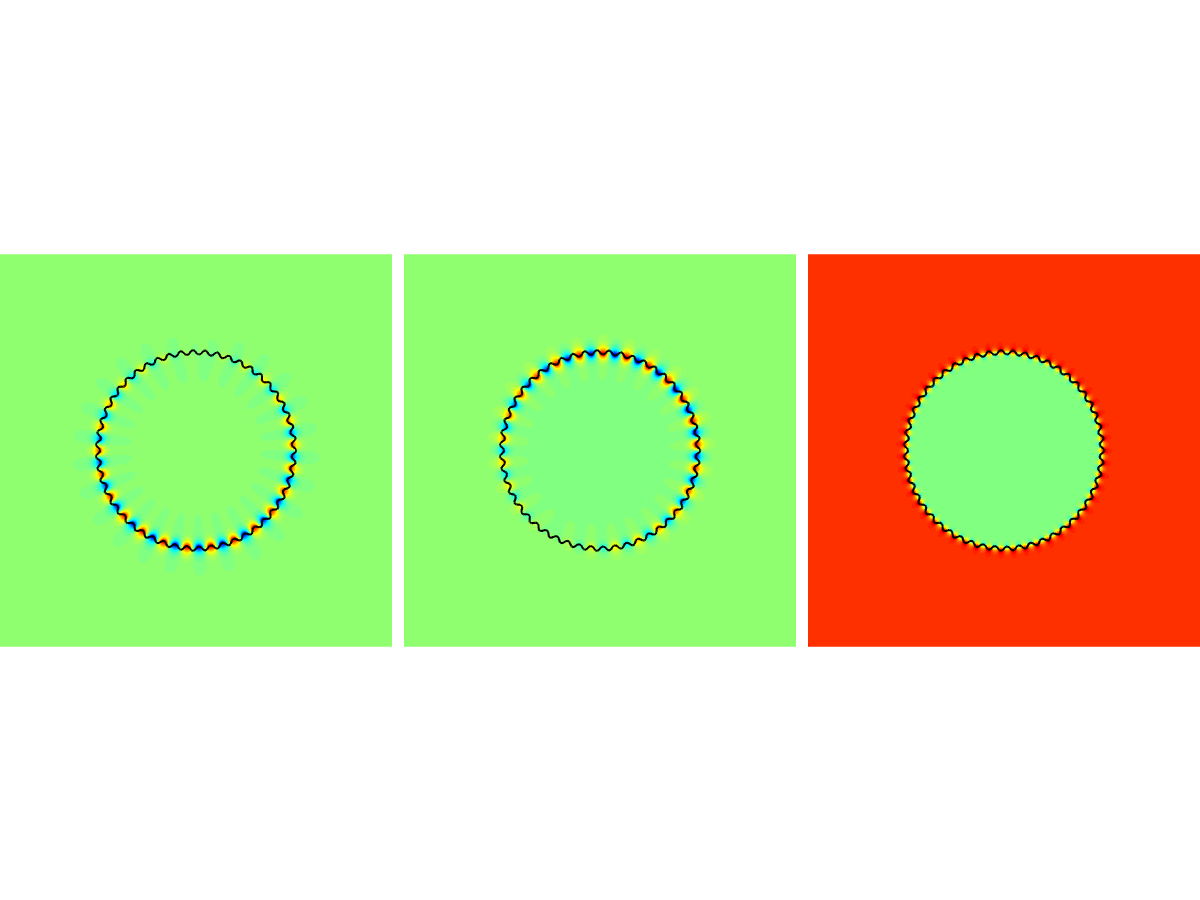} 
\vspace{.1cm}

\includegraphics[width=1\linewidth,trim={0.2cm 5.5cm 0.2cm 5.5cm},clip]{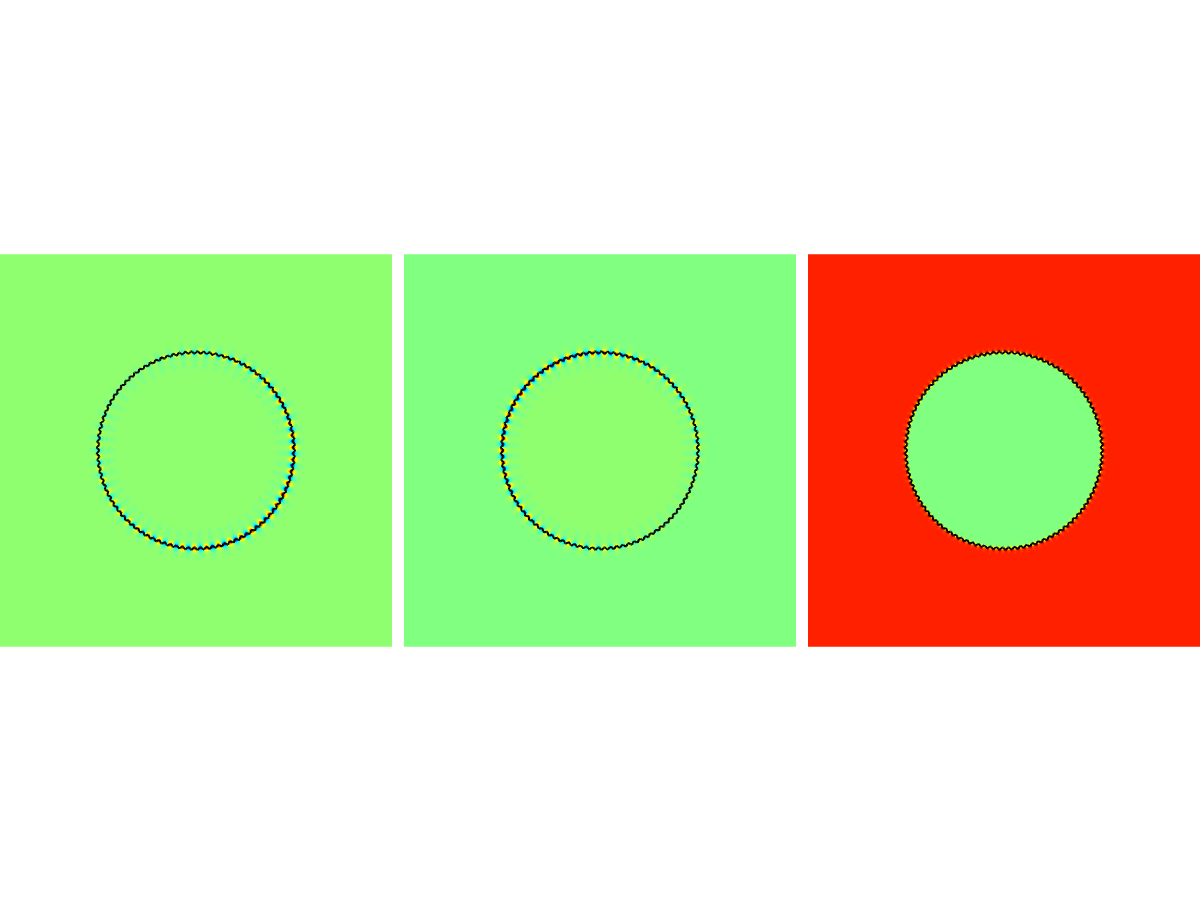} 
\caption{$p$, $p+1$, and $p+2$  Steklov eigenfunctions of $\Omega^{p\star}$   for $p=51$ (top) and $p=101$ (bottom).}
\label{fig:eigFunp51101}
\end{center}
\end{figure}

\section{Discussion} \label{sec:disc}
In this paper, we developed a computational method for extremal Steklov eigenvalue problems and applied  it to  study the problem of maximizing the $p$-th Steklov eigenvalue as a function of the domain with a volume constraint. The optimal domains,  spectrum, and eigenfunctions  are very structured, in contrast with other extremal eigenvalue problems. There are several interesting directions for this work. The first is to use conformal or quasiconformal maps to better understand the optimal domains (see \cite{Girouard2015}). It would be very interesting to extend these computational results to higher dimensions and see if the optimal domains there are also structured. The computational methods developed here could also be used to investigate other functions of the Steklov spectrum. Finally,  the optimal domains in this work could potentially find  application in sloshing problems, where it is desirable to engineer a vessel to have a  large spectral gap to avoid certain exciting frequencies \cite{troesch1965}.

\bigskip
\subsection*{Acknowledgements} We would like to thank Dorin Bucur for pointing us towards \cite{dambrine2014extremal} and Oscar Bruno and Nilima Nigam for collaboration in building the eigenvalue solver.

\clearpage
{\footnotesize 
\bibliographystyle{amsalpha}
\bibliography{ExtremeSteklovArXiv.bib} }

\end{document}